\documentclass[reqno, 12pt]{amsart}
\usepackage{array}
\usepackage{amsmath}
\usepackage{amsfonts}
\usepackage{amssymb}
\usepackage{mathrsfs}
\usepackage{enumerate}
\usepackage{amsthm}
\usepackage{verbatim}
\usepackage{amsmath, amscd}
\usepackage[usenames,dvipsnames]{color}
\usepackage{bm}
\usepackage{xy}
\xyoption{all}
\usepackage{color}
\usepackage{marvosym}
\usepackage{amsbsy}

\usepackage{hyperref}
\usepackage{tikz}
\usepackage{marginnote}
\usetikzlibrary{matrix,arrows,backgrounds}

\newtheorem{theorem}{Theorem}[section]

\newtheorem{lemma}[theorem]{Lemma}
\newtheorem{proposition}[theorem]{Proposition}
\newtheorem{corollary}[theorem]{Corollary}

\newtheorem{question}[theorem]{Question}

\theoremstyle{definition}
\newtheorem{definition}[theorem]{Definition}
\newtheorem{remark}[theorem]{Remark}

\newtheorem{example}[theorem]{Example}
\newtheorem{caution}[theorem]{Caution}

\newcommand{\op}[1]{\operatorname{#1}}

\newcommand{\newterm}{\textsf}

\newcommand{\dbcoh}[1]{\operatorname{D}^{\operatorname{b}}(\operatorname{coh }#1)}

\newcommand{\dsing}[1]{\operatorname{D}_{\operatorname{sg}}(#1)}

\newcommand{\gm}{\mathbb{G}_m}

\renewcommand{\div}{\operatorname{div}}

\newcommand{\I}{\mathcal I}
\newcommand{\J}{\mathcal J}
\newcommand{\cone}{\operatorname{Cone}}

\def\Z{\op{\mathbb{Z}}}
\def\C{\op{\mathbb{C}}}
\def\R{\op{\mathbb{R}}}
\def\Q{\op{\mathbb{Q}}}

\def\O{\op{\mathcal{O}}}
\def\A{\op{\mathbb{A}}}
\def\P{\op{\mathbb{P}}}

\def\T{\op{\mathcal{T}}}

\def\aut{\operatorname{Aut}}
\def\diag{\operatorname{diag}}

\def\tif{\text{if } }

\def\tand{\text{ and } }

\begin{document}

\title{Derived Categories of BHK Mirrors}

\author{David Favero}
\address{
    University of Alberta, Department of Mathematics,  Edmonton, AB Canada \newline
  Korean Institute for Advanced Study, Seoul, South Korea}
   \email { favero@ualberta.ca} 

\author{Tyler L. Kelly}
\address{
University of Cambridge, DPMMS, Cambridge CB3 0WB, United Kingdom \newline
    University of Birmingham, School of Mathematics, Edgbaston, Birmingham B15 2TT, United Kingdom}
   \email{ t.kelly.1@bham.ac.uk}

\numberwithin{equation}{section}

\begin{abstract}
We prove a derived analogue to the results of Borisov, Clarke, Kelly, and Shoemaker on the birationality of Berglund-H\"ubsch-Krawitz mirrors. Heavily bootstrapping off work of Seidel and Sheridan, we obtain Homological Mirror Symmetry for Berglund-H\"ubsch-Krawitz mirror pencils to hypersurfaces in projective space.	
\end{abstract}

\maketitle
\tableofcontents

 \section{Introduction}

In 1989, Candelas, Lynker, and Schimmrigk wrote a prophetic paper with computer-based evidence of a mathematical phenomenon predicted by string theorists.  Their paper provides a list of Calabi-Yau hypersurfaces in weighted-projective 4-space which mostly partner off.  Namely, if there is a Calabi-Yau threefold with Hodge numbers $(h^{1,1}, h^{2,1})$ on the list then there is often one with the Hodge numbers flipped: $(h^{2,1}, h^{1,1})$ \cite{CLS90} - the so called mirror.
Greene and Plesser followed with a physical construction of the mirror partners to Fermat hypersurfaces in weighted-projective spaces  \cite{GP90}.

The next generalization was provided by  Berglund and H\"ubsch  \cite{BH93}.
The Berglund-H\"ubsch construction provides a mirror for quasismooth hypersurfaces in a weighted-projective space.   One takes a polynomial
$$
F_A := \sum_{i=0}^n \prod_{j=0}^n x_j^{a_{ij}}
$$
associated to an invertible matrix $A = (a_{ij})$ which defines a quasismooth hypersurface  in weighted projective space $\mathbf{P}(q_0,\ldots,q_n)$. Its mirror is roughly the hypersurface given by the transposed polynomial
$$
F_{A^T} : = \sum_{i=0}^n \prod_{j=0}^n x_j^{a_{ji}}
$$
in another weighted projective space.  More precisely, one takes additional quotients on both sides by finite groups which correspond to an exchange of the geometric and quantum symmetries of the polynomials $F_A$ and $F_{A^T}$.

This proposal had its limitations.  For example, it was unable to accommodate the latest theory seen in a paper of Candelas, de la Ossa, and Katz \cite{CdK95}.
 Fortunately, a toric mirror construction due to Batyrev \cite{Bat} saved the day.  Batyrev's mirror construction was extended to Calabi-Yau complete intersections by Batyrev and Borisov the following year, providing a pivotal construction for future work on mirror symmetry.

In 2007,  Berglund-H\"ubsch mirrors  resurfaced in a series of articles after Fan, Jarvis, and Ruan used the Berglund-H\"ubsch construction to explain the self-duality of $A_n$ and $E_n$ singularities and study Landau-Ginzburg mirror symmetry \cite{FJR13}. Soon afterward, Krawitz gave a well-defined version of Berglund-H\"ubsch mirror symmetry \cite{Kr09} and Chiodo and Ruan \cite{CR11} went on to prove that the Berglund-H\"ubsch-Krawitz (BHK) mirrors form a mirror pair on the level of Chen-Ruan orbifold cohomology \cite{CR04} (and consequently stringy cohomology).

At this point, both Batyrev-Borisov mirrors and Berglund-H\"ubsch-Krawitz mirrors had evidence of being correct mirrors; however, given a Calabi-Yau hypersurface that has both a Batyrev-Borisov mirror and a BHK mirror, these mirrors may not be isomorphic.  To make matters worse, varying certain choices involved in either construction can result in multiple mirrors.  What to do?

As it turns out, this phenomenon is not so mysterious.
In the physics literature, it is a well-studied story about different phases or energy limits of the mirror.  Meanwhile in the math literature, we have a more specific ansatz:  the paper of Clarke \cite{Cla08} which unifies the constructions of Givental, Hori-Vafa, Berglund-H\"ubsch, and Batyrev-Borisov, together with Kontsevich's Homological Mirror Symmetry Conjecture.

In light of Kontsevich's Homological Mirror Symmetry Conjecture, a mirror pair of Calabi-Yau manifolds $\mathcal{M}$ and $\mathcal{W}$ should exchange symplectic and complex data at the level of categories. Namely, the Fukaya category of $\mathcal{M}$ (the A-model) should be equivalent to the bounded derived category of coherent sheaves of its mirror $\mathcal{W}$ (the B-model), i.e., 
\[
\op{Fuk}(\mathcal{M}) \cong \dbcoh{\mathcal{W}} \tand \op{Fuk}(\mathcal{W}) \cong \dbcoh{\mathcal{M}} .
\]

Consider a Calabi-Yau manifold $\mathcal M$.  As a consequence of the Homological Mirror Symmetry Conjecture, the derived category of its mirror should depend neither  on the construction of the mirror nor on the complex structure of $\mathcal M$.  In summary, if we have multiple mirrors $\mathcal W_1, ..., \mathcal W_r$ that arise from various choices of complex structure on $\mathcal M$ or mirror constructions, then we expect that these mirrors have equivalent derived categories
\[
\op{Fuk}(\mathcal{M}) \cong  \dbcoh{\mathcal{W}_1} \cong ... \cong  \dbcoh{\mathcal{W}_r}.
\]

In this paper, we prove that this is precisely the case for Berglund-H\"ubsch-Krawitz mirrors in Gorenstein Fano toric varieties.  The proof utilizes the fact that a variation of complex/algebraic structure is mirrored by a variation of K\"ahler structure specifically realized through variation of Geometric Invariant Theory quotients.  Morally, this allows us to apply the work of Ballard-Favero-Katzarkov \cite{BFK12} to obtain the desired derived equivalence.  However, a modification using partial compactifications of toric vector bundles is necessary to realize this to fruition. We will now provide a more precise mathematical explanation of our results.

\subsection{Precise Results}

Let us fix once and for all, $\kappa$, an algebraically closed field of characteristic $0$.  We work strictly over such a field.

The context of BHK mirror symmetry consists of taking a polynomial
$$
F_A : = \sum_{i=0}^n \prod_{j=0}^n x_j^{a_{ij}}
$$
where the matrix $A:=(a_{ij})$ is invertible and the polynomial $F_A$ cuts out a quasismooth Calabi-Yau hypersurface in some weighted-projective stack $\P(q_0,\ldots, q_n)$. Then one takes a group $G$ that is a subset of the group of diagonal automorphisms
$$
\op{Aut}(F_A) = \{ (\lambda_i) \in (\gm)^{n+1} | F_A(\lambda_i x_i) = F_A(x_i)\}
$$
so that $G$ acts trivially on holomorphic $(n,0)$ forms of $Z(F_A)$. We take the quotient stack
$$
Z_{A,G} = \left[ \frac{\{F_A = 0\} }{G \gm}\right] \subseteq \left[\frac{\A^{n+1} \setminus\{0\}}{G\gm}\right] = \frac{\P(q_0, \ldots, q_n) }{\overline{G}}
$$
where $\gm$ acts with weights $q_0, ..., q_n$ and $\overline{G} := G / (G \cap \gm)$. BHK mirror symmetry proposes a mirror that is associated to the transposed polynomial
$$
F_{A^T} : = \sum_{i=0}^n \prod_{j=0}^n x_j^{a_{ji}}.
$$
The polynomial $F_{A^T}$ cuts out a quasismooth Calabi-Yau hypersurface in another weighted-projective stack $\P(r_0, \ldots, r_n)$. Krawitz \cite{Kr09} identified the dual group $G^T_A$ (see Equation~\eqref{DualGroup}) which depends on both $G$ and $A$ so that one can state the BHK mirror to be,
$$
Z_{A^T, G^T} : = \left[ \frac{\{F_{A^T} = 0\} }{G^T_A \gm}\right] \subseteq \left[\frac{\A^{n+1} \setminus\{0\}}{G^T_A \gm}\right] = \frac{\P(r_0, \ldots, r_n) }{\overline{G^T_A}}.
$$
Chiodo and Ruan \cite{CR11} proved the following.
\begin{theorem}[Chiodo-Ruan]
On the level of Chen-Ruan cohomology, the Hodge diamonds for $Z_{A,G}$ and its BHK mirror $Z_{A^T, G^T}$ flip:
$$
H^{p,q}_{\text{CR}} (Z_{A,G}, k) \cong H^{n-1-p,q}_{\text{CR}}(Z_{A^T,G^T_A}, k).
$$
\end{theorem}
This is the analogous result to that of Batyrev and Borisov for their construction. One can ask how this construction compares to the mirror construction of Batyrev for hypersurfaces of Fano toric varieties. The answer is that the mirror construction matches if and only if the polynomial $F_A$ is a Fermat variety in a (necessarily Gorenstein) Fano toric variety. In fact, if one starts with a non-diagonal polynomial $F_A$ sitting in a (possibly Fano) toric variety, very often one gets a BHK mirror $Z_{A^T, G^T_A}$ that is in a non-Gorenstein (and consequently non-Fano) toric variety (see Example \ref{NonGorMirror}).  Such a BHK mirror $Z_{A^T, G^T_A}$ does not have a mirror prescribed by Batyrev and Borisov, and consequently does not match up to the varieties prescribed to be the Batyrev mirror.  

We can also consider two polynomials $F_A$ and $F_{A'}$ that have the same weights. We can then consider a group $G\subset \op{Aut}(F_A)\cap \op{Aut}(F_{A'})$ that acts trivially on the holomorphic forms of both hypersurfaces $Z(F_A)$ and $Z(F_{A'})$ and consequently get two quotient stacks $Z_{A, G}$ and $Z_{A', G}$ in the same toric variety; however, their respective BHK mirrors $Z_{A^T, G_A^T}$ and $Z_{(A')^T, G_{A'}^T}$ may be in completely different toric varieties.  This leads to the following question of Iritani:

\begin{question}[Iritani]
Given two quotient stacks $Z_{A,G}$ and $Z_{A', G}$ that sit in the same toric variety, are their BHK mirrors $Z_{A^T, G^T_A}$ and $Z_{(A')^T, G^T_{A'}}$ birationally equivalent?
\end{question}

This question is answered affirmatively in many ways in the literature by Borisov \cite{Bor13}, Shoemaker \cite{Sho14}, Kelly \cite{Kel13}, and Clarke \cite{Cla13}. 
In this paper, we prove that these mirrors are the same from the perspective of homological mirror symmetry.

\begin{theorem}[=Theorem \ref{thm: BHK main result}]\label{introBHK}
Given two quotient stacks $Z_{A,G}$ and $Z_{A', G}$ that sit in the same Gorenstein Fano toric variety, their BHK mirrors $Z_{A^T, G^T}$ and $Z_{(A')^T, G^T_{A'}}$ are derived equivalent.
\end{theorem}

By joining this theorem with the main theorem of \cite{FK14}, we can say the following:  given a Calabi-Yau complete intersection or hypersurface in a Gorenstein Fano toric variety, there may be various distinct ways to construct its mirror using Berglund-H\"ubsch-Krawitz or Batyrev-Borisov mirrors, but all of these mirrors are derived equivalent.

Moreover, when proving Theorem \ref{introBHK}, one gets derived equivalences amongst members of families of hypersurfaces in the different weighted-projective stacks. A priori, Berglund and H\"ubsch proposed their mirror duality as a construction for specific Calabi-Yau hypersurfaces.  We can insert these specfic hypersurfaces into a family and explicitly match each member of this extended family of Calabi-Yau varieties to one another pointwise by derived equivalence.

The most basic extension to families allows one to apply Polishchuk-Zaslow, Seidel, and Sheridan's proof of Homological Mirror Symmetry for Calabi-Yau hypersurfaces in projective space \cite{PZ, Sei03, Sheridan}.  Since the Polishchuk-Zaslow result (dimension $1$) is analogous but slightly different to state, we treat the cases of Seidel (dimension $2$) and Sheridan (dimension $\geq 3$) which one can do simultaneously.

Namely, let $\Lambda$ be the universal Novikov field which contains $\C[[r]] \subseteq \Lambda$ so that $r$ is a formal parameter.
Over the universal Novikov field, we define a \newterm{Berglund-H\"ubsch-Krawitz pencil} as
$$
Z_{A, G}^{\op{pencil}} : = \left[ \frac{\{x_0 ... x_n + rF_{A} = 0\} }{G \gm}\right] \subseteq \left[\frac{\A^{n+1} \setminus\{0\}}{G \gm}\right] = \frac{\P(q_0, \ldots, q_n) }{\overline{G}}
$$
where
\[
\P(r_0,\ldots, r_n) := [\mathbb A^{n+1} \backslash \{0\} / \gm]
\]
is a \newterm{weighted projective stack}.
For Berglund-H\"ubsch-Krawitz pencils we have the following.
\begin{theorem}[=Theorem~\ref{HMS}]
\label{introHMS}
Homological Mirror Symmetry holds for Berglund-H\"ubsch-Krawitz mirror pencils in projective space over the universal Novikov field.

  More precisely, if $F_A$ defines a smooth hypersurface in complex projective space $\C\P^n$ (in particular $G=\Z_{n+1}$)   with $n \geq 3$, there is an equivalence of triangulated categories
\[
\op{Fuk}{Z_{A, G}} \cong \dbcoh{Z_{A^T, G^T_A}^{\op{pencil}}}.
\]
\end{theorem}

\subsection{Plan of the Paper}
Here is a brief summary of how the paper is organized.

In Section~\ref{sec: Background}, we outline BHK mirror symmetry, give a toric reinterpretation due to Borisov and Shoemaker, and define the multiple mirrors that we will prove are derived equivalent.

In Section~\ref{sec: categories}, we provide background on the category of singularities and in particular the theorems of Orlov, Isik, and Shipman which we will use.

 In Section~\ref{sec: algebraic}, we prove criteria for derived equivalences for complete intersections that are zero loci of sections of different vector bundles. This is placed in the context of equivalences of categories of singularities amongst various partial compactifications of vector bundles, and we show how the latter follows from some recent results on variations of GIT quotients. 

 In Section~\ref{sec: BHK}, we apply our framework to prove the derived analogue to the birationality result of Borisov, Clarke, Shoemaker, and the second-named author on BHK mirrors. We then discuss this in an explicit example.

\vspace{2.5mm}
\noindent \textbf{Acknowledgments:}
We heartily thank Colin Diemer for suggesting that VGIT may relate to the BHK picture and give special thanks to Charles Doran for input on this project from start to finish.  The first-named author is grateful to the Korean Institute for Advanced Study for their hospitality while this document was being prepared and especially to Bumsig Kim for insightful conversations on Clarke's mirror construction.   The second-named author thanks the Pacific Institute for the Mathematical Sciences for its hospitality in his visits as they expedited the progress of this work.
This project was also greatly aided by stimulating conversations and suggestions from many great mathematicians; Matthew Ballard, Ionut Ciocan-Fontanine, Ron Donagi, Daniel Halpern-Leistner, Yuki Hirano, and Xenia de la Ossa.  The authors would also like to extend their gratitude to the referees for their thorough readings and many improvements to  this manuscript.

The first-named author is grateful to the Natural Sciences and Engineering Research Council of Canada for support provided by a Canada Research Chair and Discovery Grant (CRC TIER2 229953 and RGPIN 04596). The second-named author acknowledges that this paper is based upon work supported by the National Science Foundation under Award No. DMS-1401446 and the Engineering and Physical Sciences Research Council under Grant EP/N004922/1.

 \vspace{2.5mm}

 \section{Background}
 \label{sec: Background}

 \subsection{Berglund-H\"ubsch-Krawitz Mirror Symmetry}

 Let
 \begin{equation}
 F_A = \sum_{i=0}^n \prod_{j=0}^n x_j^{a_{ij}}, \quad a_{ij} \geq 0
\label{eq: FA}
 \end{equation}
 be a polynomial equation that is the sum of $n+1$ monomials in $n+1$ variables and set the matrix $A : = (a_{ij})_{i,j=0}^n$.
  We impose the following conditions:

 \begin{definition}\label{KSPoly}
 The polynomial $F_A$ above is  a \newterm{Kreuzer-Skarke} polynomial if:
 \begin{enumerate}[a)]
 \item the matrix $A$ is invertible over $\Q$;
 \item there exists positive integers $q_i$ so that the sum $\sum_j q_j a_{ij}$ is independent of $i$; and
 \item when viewed as a polynomial map, $F_A: \A^{n+1} \rightarrow \A$ has exactly one critical point.
 \end{enumerate}
 \end{definition}

\begin{remark}
 These conditions are restrictive.  Their classification is discussed in Section~\ref{sec: KS cleaves}.
 \end{remark}

 We then can look at the well-defined hypersurface in a weighted projective stack that is cut out by the polynomial $F_A$,
 $$
 Z_A : = \left\{ F_A = 0\right\} \subseteq \P{(q_0, \ldots, q_n)}.
 $$
Condition (b) implies that the hypersurface is well-defined in this weighted projective space and condition (c) implies that the hypersurface is quasismooth. We further impose the condition that $Z_A$ is Calabi-Yau. This is equivalent to the condition that the degree of the polynomial $F_A$ is the sum of the weights $\sum_i q_i$. This in turn is equivalent to the condition that the sum of the entries in the inverse matrix $A^{-1}$ is one, i.e., $\sum_{i,j} (A^{-1})_{ij} = 1$. If we want that the hypersurface $Z_A$ to be a generalized Calabi-Yau, i.e., a Fano variety with a middle Hodge structure similar to that of a Calabi-Yau manifold (see Definition 2.2 of \cite{FIK}), then we merely desire that the sum of the entries of the inverse matrix $A^{-1}$ sums to an integer.

These hypersurfaces are highly symmetric. If we take the the torus $(\gm)^{n+1} $ acting coordinatewise on $\P(q_0, \ldots, q_n)$, we can describe many subgroups of the torus that represent certain symmetries of the polynomial $F_A$ and the hypersurface $Z_A$. Consider the group $\aut_{\diag}(F_A)$ of diagonal symmetries rescaling the coordinates and preserving $F_A$,
\begin{equation}
\aut_{\diag}(F_A) = \left\{ (\lambda_i) \in (\mathbb{G}_m)^{n+1} \middle| \ F_A(\lambda_ix_i) = F_A(x_i) \right\}.
\end{equation}
This group is generated by the elements $\rho_j = (\exp(2\pi i a^{j0}), \ldots, \exp(2\pi i a^{jn}))$, where $a^{ij} := (A^{-1})_{ij}$.

In the case where $Z_A$ is a Calabi-Yau variety, not all the elements in the group of diagonal symmetries leave the unique (up to scaling) holomorphic form invariant, hence we define a subgroup,
\begin{equation}
SL(F_A) = \left\{ (\lambda_i) \in \aut_{\diag}(F_A)  \middle| \ \prod_i \lambda_i = 1\right\}
\end{equation}
of elements that, when viewed a diagonal matrix acting on the coordinates $x_i$ has determinant one.

Some of these symmetries of $F_A$ act trivially on the hypersurface $Z_A$. In particular, one has the \newterm{exponential grading operator} subgroup,
$$
J_{F_A} = \langle \rho_0\cdots \rho_n\rangle \subseteq \aut_{\diag}(F_A),
$$
which acts trivially on the hypersurface $Z_A$.  Take a group $G$ so that
\begin{equation}\label{groupGdef}
J_{F_A} \subseteq G \subseteq SL(F_A)
\end{equation}
 and denote by $\overline G$ the quotient $G/ J_{F_A}$. If we start with a Calabi-Yau hypersurface $Z_A$, when we quotient by $\overline G$ we get a Calabi-Yau orbifold $Z_{A,G} : = [Z_A / \overline G]$. Alternatively, we may view this as a (smooth) Deligne-Mumford global quotient stack,
\begin{equation}
Z_{A,G}= \left[\frac{ \{F_A = 0\}}{G\gm}\right] \subseteq \left[\frac{\A^{n+1} \setminus \{0\} }{G\gm}\right] = \frac{\P(q_0,\ldots, q_n)}{\overline{G}}.
\end{equation}

 Berglund-H\"ubsch-Krawitz mirror symmetry provides a mirror for this orbifold in the following way. We define the transposed polynomial,
 \begin{equation}
 F_{A^T} = \sum_{i=0}^n \prod_{j=0}^n x_j^{a_{ji}},
 \label{eq: FAT}
 \end{equation}
and the transposed group,
\begin{equation}\label{DualGroup}
G^T_A = \left\{ \prod_j (\rho_j^T)^{s_j} \middle| \ \text{ $\prod_j x_j^{s_j}$ is $G$-invariant}\right\},
\end{equation}
where  $\rho_j^T : = ((\exp(2\pi i a^{0j}), \ldots, \exp(2\pi i a^{nj}))$. For a more functorial description of $G^T_A$, we refer the interested reader to Definition 2.7 of \cite{Cla13}. Provided $F_A$ and $G$ above, we enjoy the following properties about their transposed counterparts:
\begin{enumerate}[i.]
\item $F_{A^T}$ is a Kreuzer-Skarke polynomial, but with possibly different weights $r_i$.
\item If $J_{F_A} \subseteq G$, then $G^T_A \subseteq SL(F_{A^T})$.
\item If $G \subseteq SL(F_A)$, then $J_{F_{A^T}} \subseteq G^T_A$.
\item The hypersurface $Z_{A^T} : = \{ F_{A^T} = 0\} \subseteq \P(r_0, \ldots, r_n)$ is (Fano) Calabi-Yau if $Z_A$ is (Fano) Calabi-Yau.
\end{enumerate}

Denote by $\overline{G^T_A}$ the quotient $G^T_A / J_{F_{A^T}}$. If we start with a Calabi-Yau hypersurface $Z_A$ and a group $G$ so that $J_{F_{A^T}} \subseteq G \subseteq SL(F_A)$, we obtain the quotient stack,
\begin{equation}
Z_{A^T, G^T} =\left[\frac{\{F_{A^T} = 0\}}{G^T_A \gm}\right] \subseteq \left[ \frac{ \A^{n+1} \setminus \{0\} }{G^T_A\gm}\right] = \frac{\P(r_0,\ldots, r_n)}{\overline{ G^T_A}},
\end{equation}
that is also a Calabi-Yau orbifold.

\begin{example}\label{FermatQuintic}
If one takes $A$ to be the $5\times 5$ diagonal matrix, $A=5I_5$, then one gets the Fermat polynomial $F_{5I_5} = x_0^5+x_1^5 + x_2^5+x_3^5+x_4^5$ which carves out the Fermat hypersurface $X_{5I_5} \subseteq \P^4$. Take the group $G$ to be the exponential grading operator $J_{F_{5I_5}}$ so that we are looking at the Fermat quintic threefold $Z_{A,G} = X_{5I_5}$. BHK mirror symmetry predicts the mirror $Z_{A^T, G^T_A} = X_{5I_5} / (\Z_5)^3 \subseteq \P^4 / (\Z_5)^3$ where the $(\Z_5)^3$ acts coordinatewise by the generators $(\zeta, \zeta^{-1}, 1,1,1)$, $(\zeta, 1, \zeta^{-1}, 1,1)$, and $(\zeta,1,1,\zeta^{-1},1)$ where $\zeta$ is a primitive fifth root of unity. This is the same mirror hypersurface that is predicted by Greene-Plesser and Batyrev.
\end{example}

\begin{example}\label{NonGorMirror}
Suppose one takes $A'$ to be the matrix of exponents for the polynomial $F_{A'} = x_0^4x_1 +x_1^4x_2 + x_2^4x_3 + x_3^4x_4 + x_4^5$, which carves out a quintic hypersurface $Z_{A'} \subseteq  \P^4$. As before, take the group $G$ to be the exponential grading operator. In this case, $\overline{G}_{A'}^T = \{\text{id}\}$ and BHK mirror symmetry predicts the mirror $Z_{(A')^T, G^T_{A'}} = Z(y_0^4 + y_0y_1^4+y_1y_2^4 + y_2y_3^4 + y_3y_4^5) \subseteq \P^4(64,48,52,51,41)$.  The hypersurface $Z_{A^T, G^T_A}$ is not predicted by Greene-Plesser or Batyrev, rather, it does not sit in a Gorenstein Fano toric variety. Hypersurfaces in non-Gorenstein toric varieties do not have mirror constructions due to any of the naturally toric mirror constructions created. For more examples of BHK mirrors to projective hypersurfaces, consult Tables 5.1-3 of \cite{DG11}.
\end{example}

The mirrors in Examples \ref{FermatQuintic} and \ref{NonGorMirror} do not have an obvious relation.
The BHK mirror construction does not predict the same mirror for two (symplectomorphic) hypersurfaces $Z_{A,G}$ and $Z_{A', G}$ that sit in the same toric variety. However, the question of if $Z_{A^T, G^T_A}$ and $Z_{(A')^T, G^T_{A'}}$ are birational has been well-studied recently by many approaches. The theorem below states a relevant amalgamation of these results (which is not described in full generality):

 \begin{theorem}[\cite{Bor13, Sho14, Kel13, Cla13}]
Take two polynomials $F_A$ and $F_{A'}$ as above so that the Calabi-Yau hypersurfaces $Z_A$ and $Z_{A'}$ are hypersurfaces in the same weighted projective space $\P(q_0, \ldots, q_n)/G$ where $J_{F_A} = J_{F_{A'}} \subseteq G  \subseteq SL(F_A) \cap SL(F_{A'})$.  One then has two CY orbifolds $Z_{A,G}$ and $Z_{A',G}$ as hypersurfaces in the orbifold $\P(q_0, \ldots, q_n) / (G/ J_{F_A})$. The BHK mirrors $Z_{A^T,G^T_A}$ and $Z_{(A')^T, G^T_{A'}}$ are birational.
\label{thm: birationality BHK}
\end{theorem}

  In the following sections, we mesh the many approaches to this question with variational geometric invariant theory in order to prove a result more in line with Kontsevich's homological mirror symmetry---derived equivalence.

\subsection{Toric reinterpretation of BHK mirrors}\label{toricreinterpretation}

There have been a few toric reinterpretations of BHK mirror duality in the literature (\cite{Bor13}, \cite{Cla08},  \cite{Sho14}). In this subsection, we will give a brief overview of the framework that we will use and introduce the relevant notation for the BHK mirror construction both in a Landau-Ginzburg and a Calabi-Yau setting.

\begin{caution}
In this section, we use the notation $\overline M$ for a lattice which in \cite{Bor13} is denoted simply by $M$.  The notation $\overline M$ is often used in the literature for $M \oplus \Z^n$.  In our case, we will consider just $\overline M = M \oplus \Z$ in Section~\ref{sec: BHK}. In particular, we will specialize to the setting where $\op{deg} = (0,1)$ and $\overline M / \op{deg} = M$. 
\end{caution}

We start with the setup of \cite {Bor13}.  Take two free abelian groups ${\overline M_0}$ and ${\overline N_0}$ with bases $\{ u_i\}$ and $\{v_i\}$, respectively. Consider a matrix $A=(a_{ij})_{i,j=0}^n$ that is the matrix associated to a Kreuzer-Skarke polynomial as defined in Definition~\ref{KSPoly}.  Now, we define a pairing $\langle,\rangle: {\overline M_0} \times {\overline N_0} \rightarrow \Z$ between the two free abelian groups ${\overline M_0}$ and ${\overline N_0}$ by imposing that $\langle u_i, v_j\rangle= a_{ij}$. Choose overlattices
${\overline M}$ and ${\overline N}$ so that ${\overline M}$ and ${\overline N}$ are dual to one another and we have the following containments:
\begin{equation}\label{overlattices}
{\overline N_0} \subseteq {\overline N} \subseteq {\overline M_0}^\vee; \qquad {\overline M_0} \subseteq {\overline M} \subseteq {\overline N_0}^\vee.
\end{equation}
We then have exact sequences
\begin{equation}\begin{aligned}
0 \rightarrow & {\overline M} \rightarrow {\overline N_0}^\vee \rightarrow {\overline N_0}^\vee / {\overline M} \rightarrow 0;\\
&m \mapsto \langle \overline m , - \rangle,
\end{aligned}\end{equation}
and
\begin{equation}\begin{aligned}
0 \rightarrow & {\overline N} \rightarrow {\overline M_0}^\vee \rightarrow {\overline M_0}^\vee / {\overline N} \rightarrow 0;\\
&\overline n \mapsto \langle -, \overline n \rangle.
\end{aligned}\end{equation}
The first map is the toric divisor map $\div$ for the toric variety $(\kappa \otimes {\overline N_0}) / ({\overline N_0}^\vee / {\overline M})$ with ray generators $v_i$, as it can be written
\begin{equation}
\overline m \mapsto \sum_i \langle \overline m, v_i\rangle v_i ^\vee.
\end{equation}

The second map is the monomial map $\operatorname{mon}$ for the rational function $\sum_i x^{u_i}$ as it can be written
\begin{equation}
\overline n \mapsto \sum_i \langle u_i, \overline n \rangle u_i ^\vee.
\end{equation}
This gives us a pair consisting of a space and a function
\begin{equation}
\left((\kappa \otimes {\overline N_0}) / ({\overline N_0}^\vee / {\overline M}); \sum x_i ^{u_i}\right),
\end{equation}
often referred to as a \newterm{Landau-Ginzburg (LG) model}.

Following Clarke \cite{Cla08}, the mirror LG model is given by swapping ${\overline M}$ and ${\overline N}$ and the maps $\operatorname{mon}$ and $\div$.  Hence, in this setting, the mirror is the pair
\begin{equation}
(\kappa \otimes {\overline M_0}) / ({\overline M_0}^\vee / {\overline N}); \qquad \sum x_i ^{v_i}.
\end{equation}

Notice that we have a $\Z$-basis for ${\overline N_0}$, namely $\{v_i\}$, so we have natural functions on the semi-group ring $\kappa[{\overline N_0}]$ given by the $v_i^\vee$. We denote these functions by $x_i$. In these coordinates, we write the monomial $x^{u_i}$ as
\begin{equation}\label{latticepointsmonomial}
x^{u_i} = \prod_j x_j^{\langle u_i, v_j\rangle} = \prod_j x_j^{a_{ij}},
\end{equation}
hence
\begin{equation}
\sum_i x^{u_i} = \sum_i \prod_j x_j^{a_{ij}} = F_A
\end{equation}
(see Equation~\eqref{eq: FA}).  This gives a more intrinsic description of the function $F_A$.

Analogously, we take the natural functions on the semi-ring $\kappa[{\overline M_0}]$ given by the dual elements $u_i^\vee$. We denote these functions by $y_i$. In these coordinates, we write the monomial $x^{v_i}$ as
\begin{equation}\label{mirmoninterp}
x^{v_j} =  \prod_i y_i^{\langle u_i, v_j\rangle} = \prod_i x_i^{a_{ij}},
\end{equation}
hence
\begin{equation}\label{transposedpoltoric}
\sum_j x^{v_j} = \sum_j \prod_i x_i^{a_{ij}} = F_{A^T}
\end{equation}
(see Equation~\eqref{eq: FAT}).
We have now checked that the polynomials in this toric interpretation match to the original construction.

The groups involved in the Berglund-H\"ubsch-Krawitz construction also have a toric interpretation. The choice of overlattices $ \overline M$ and $ \overline N$ corresponds to the choice of group. 
First, consider the unique elements $\deg \in {\overline N_0}^\vee$ and $\deg^\vee \in {\overline M_0}^\vee$ so that
$$
\langle \deg, v_i\rangle = \langle u_i, \deg^\vee\rangle = 1 \text{ for all $i$}.
$$

\begin{proposition}[Propositions 2.3.1 and 2.3.4 of \cite{Bor13}]\label{BorProp}
Suppose that $\sum_{i, j } (A^{-1})_{i j} \in \Z_+.$ Then for any choice of a group $G$ so that $J_{F_A} \subseteq G \subseteq SL(F_A)$, there is a pair of overlattices ${\overline M} \supseteq {\overline M_0}$ and ${\overline N}\supseteq {\overline N_0}$ so that ${\overline M}$ and ${\overline N}$ are dual lattices, $\deg \in {\overline M}$, $\deg^\vee \in {\overline N}$ and the groups ${\overline N_0}^\vee / {\overline M}$ and ${\overline M_0}^\vee / {\overline N}$ are naturally isomorphic to the groups $G$ and $G^T_A$, respectively.
\end{proposition}

While Proposition~\ref{BorProp} only requires that $\sum_{i, j } (A^{-1})_{i j} \in \Z_+,$ in the special case $\sum_{i, j } (A^{-1})_{i j} = 1$ we can in fact produce a Calabi-Yau hypersurface as follows. Take the cones $C_{{\overline M}} = \operatorname{Cone}(u_i)$ and $C_{{\overline N}} = \operatorname{Cone}(v_j)$ and produce fans $\Sigma_{{\overline M}}$ and $\Sigma_{{\overline N}}$ by taking the collection of cones that are the proper faces of the cones $C_{{\overline M}}$ and $C_{{\overline N}}$. We star subdivide each fan by the ray generated by $\deg$ and $\deg^\vee$, respectively. We then have two new fans, call them $\Sigma_{{\overline M}, \deg}$ and $\Sigma_{{\overline N}, \deg^\vee}$. 

We can now consider the toric stacks associated to these fans.  For a treatment of toric stacks that will be relevant to the proof of the derived equivalence of BHK mirrors presented here, we direct the reader to Section  5 of \cite{FK14}.  A nice boon of considering stacks here is that it applies even to the non-Gorenstein case which is prevalent in BHK mirror symmetry.

Working with stacks,  $\Sigma_{{\overline M}, \deg}$ and $\Sigma_{{\overline N}, \deg^\vee}$ always correspond to canonical bundles over the fake weighted projective stacks for BHK mirror pairs. Here, we mean that a fake weighted projective stack is the quotient stack analogous to a fake weighted projective space. This means a group quotient on $\mathbb{A}^n \setminus\{0\}$ by $G\mathbb{G}_m$ where $G$ is a finite group and $\mathbb{G}_m$ acts with positive weights on all coordinates.  To obtain this correspondence combine Section 2 of \cite{Sho14} and Proposition~\ref{prop: vec indentification} below.  Namely,  in the notation of loc.~cit.\ Section~2, the fans $\Sigma$ and $\Sigma^\vee$ correspond to the projections of $\Sigma_{{\overline M}, \deg}$ and $\Sigma_{{\overline N}, \deg^\vee}$ under the maps $\pi_{{\overline M}} : {\overline M} \rightarrow {\overline M} / (\deg)$ and $\pi_{{\overline N}}: {\overline N} \rightarrow {\overline N}/(\deg^\vee)$ respectively.

By taking the zero loci of the polynomials in \eqref{latticepointsmonomial} and \eqref{transposedpoltoric}, we obtain the Calabi-Yau orbifolds
$$
Z_{A,G} \subseteq [\P(q_0,\ldots, q_n) / \overline{G}] ; \quad Z_{A^T,G^T_A} \subseteq [\P(r_0,\ldots,r_n) / \overline{G^T_A}].
$$

\section{Categories of Singularities}
\label{sec: categories}
In this section, we provide the necessary details on categories of singularities for global quotient stacks.  We start by reminding the reader of the framework set up in Section 3 of  \cite{FK14}, and then continue with an additional observation from Orlov's original discussion of such categories \cite{Orl04}, which we require later.

Let $X$ be a variety and $G$ be an algebraic group acting on $X$.

\begin{definition}
An object of $\dbcoh{[X/G]}$ is called \newterm{perfect} if it is locally quasi-isomorphic to a bounded complex of vector bundles.  We denote the full subcategory of perfect objects by $\op{Perf}([X/G])$.  The Verdier quotient of $\dbcoh{[X/G]}$ by $\op{Perf}([X/G])$ is called the \newterm{category of singularities} and denoted by
\[
\dsing{[X/G]} := \dbcoh{[X/G]} / \op{Perf}([X/G]).
\]
\end{definition}

We now repeat Orlov's observation that the category of singularities localizes about the singular locus (Proposition 1.14 in \cite{Orl04}) in the presence of a group action.
 \begin{proposition}[Orlov]
Assume that $\op{coh}[X/G]$ has enough locally-free sheaves.  Let $i: U \to X$ be a $G$-equivariant open immersion such that  the singular locus of $X$ is contained in $i(U)$.  Then  the restriction,
   \[
    i^*: \dsing{[X/G]} \to \dsing{[U/G]},
    \]
    is an equivalence of categories.
    \label{prop: local singularity}
 \end{proposition}
 \begin{proof}
 The proof of Proposition 1.14 in \cite{Orl04} works verbatim for equivariant sheaves.
 \end{proof}
 
Our goal later on, will be to convert a problem on hypersurfaces in weighted projective space to a toric calculation.  This is done using a theorem of Isik and Shipman which allows us to pass from studying a hypersurface to studying the (toric) total space of the line bundle defining it.

The setup is general and does not involve toric varieties. Namely, consider a variety $X$ with the action of an algebraic group $G$ and an equivariant vector bundle $\mathcal{E}$ on $X$.  Take a section $s \in \op{H}^0(X, \mathcal E)^G$ and consider the zero locus $Z$ of $s$ in $X$.  The pairing with $s$ induces a global function on the total space of $\mathcal E^\vee$.    Let $Y$ be the zero locus of the pairing with $s$ in $\op{tot}(\mathcal{E}^\vee)$ i.e. the union of the zero section of $\op{tot}(\mathcal{E}^\vee)$ and $\op{tot}(\mathcal{E}^\vee|_Z)$.  Consider the fiberwise dilation action of $\gm$ on $Y$.
 \begin{theorem}[Isik, Shipman, Hirano]
 Suppose that $s$ is a regular section i.e., the Koszul complex on $s$ is a resolution of $\O_Z$.  Then there is an equivalence of categories
 \[
 \dsing{[Y/(G \times \gm)]} \cong \dbcoh {[Z/G]}.
 \]
 \label{thm: isik-shipman}
 \end{theorem}

 \begin{proof}
The theorem is originally due independently to Isik \cite{Isik} and Shipman \cite{Shipman}.  With the $G$-action, it is Proposition 4.8 of \cite{Hirano}.
\end{proof}

 \begin{corollary}
    Let $V$ be an algebraic variety with a $G \times \gm$ action.  Suppose there is an open subset $U \subseteq V$ such that $U$ is $G \times \gm$ equivariantly isomorphic to $Y$ as above and that $U$ contains the singular locus of $X$.
Then
 \[
 \dsing{[V/(G \times \gm)]} \cong \dbcoh {[Z/G]}.
 \]
 \label{cor: compact RG-flow}
 \end{corollary}
 \begin{proof}
 We have
 \begin{align*}
  \dsing{[V/(G \times \gm)]} & \cong  \dsing{[U/(G \times \gm)]} \\
  & \cong \dbcoh{[Z/G]}
 \end{align*}
where the first line is Proposition~\ref{prop: local singularity} and the second line is Theorem~\ref{thm: isik-shipman}.
 \end{proof}

\section{Torus Actions on Affine Space}
\label{sec: algebraic}

In this section, we extend the setup of Section 4 of \cite{FK14} to partial compactifications of vector bundles. Consider an affine space $X:= \mathbb A^{n+t}$ with coordinates $x_i, u_j$ for $1 \leq i \leq n, 1 \leq j \leq t$.
Let $T = \gm^{n+t}$ be the open dense torus with the standard embedding and action on $X$.
Take $S \subseteq T$ to be a subgroup and $\widetilde{S}$ be the connected component of the identity.

The possible GIT quotients for the action of $\widetilde{S}$ on $X$ \cite{MFK} have both an algebraic and toric description.  The description in terms of GIT variations comes from varying linearization on trivial bundle (which is ample as $X$ is affine).  The choice of linearization on the trivial bundle is the same thing as a choice of a character of $\widetilde{S}$.  That is, given an element $\chi \in \op{Hom}(\widetilde{S}, \gm)$, we can form the corresponding line bundle $\O(\chi)$ by pulling back the representation of $\widetilde{S}$ via the morphism of stacks $[X/\widetilde{S}] \to [\op{pt}/\widetilde{S}]$.

In studying GIT variations, it is often convenient to consider $\chi$ as an element of the vector space $\op{Hom}(\widetilde{S}, \gm) {\otimes_{\Z}} \Q$ by rationalizing denominators in order to get an equivariant line bundle. 
Now, each linearization in Mumford's GIT, or in our case, each choice of $\chi$,  determines an open subset $U_\chi$ corresponding to the semi-stable locus of $X$ with respect to $\chi$. 

Furthermore, if we think of the vector space  $\op{Hom}(\widetilde{S}, \gm) {\otimes_{\Z}} {\Q}$ as a parameter space for linearization, then it was shown in \cite{GKZ} that this parameter space has a natural fan-structure $\Sigma_{\op{GKZ}}$ called the \newterm{GKZ-fan}. 
 The fan is defined by the following property,  each $U_\chi$ is constant on the interior of each cone in the fan.     
 The maximal cones of this fan are called \newterm{chambers} and the codimension 1 cones are called \newterm{walls}.

To describe the situation precisely, apply $\op{Hom}(-, \gm)$ to the exact sequence
\[
0 \longrightarrow S  \overset{i_S}{\longrightarrow} \gm^{n+t} \overset{\text{proj}}{\longrightarrow} \op{Coker }(i_S) \to 0
\]
to get
\[
 \op{Hom}(\op{Coker }(i_S), \gm)  \overset{\widehat{\text{proj}}}{\longrightarrow} \Z^{n+t} \overset{\widehat{i_S}}{\longrightarrow} \op{Hom}(S, \gm) \to 0.
\]
Set $\nu_i(S)$ to be the element of  $\op{Hom}(\op{Coker }(i_S), \gm)^{\vee}$ given by the composition of $\widehat{\text{proj}}$ with the projection of $\Z^{n+t}$ onto its $i^{\op{th}}$ factor.  Then, we define $\nu(S)$ as the following vector
\[
\nu(S) := (\nu_1(S), ...., \nu_{n+t}(S)).
\]

\begin{theorem}
There is a bijection between chambers of the GKZ fan for the action of $S$ on $\A^{n+t}$ and regular triangulations of the set $\nu(S)$.  In particular, there are finitely many chambers of the GKZ fan.
\end{theorem}
\begin{proof}
See Proposition 15.2.9 of \cite{CLS} or Chapter 7, Theorem 1.7 of \cite{GKZ} for the original formulation in terms of polytopes.
\end{proof}

By the above, we may enumerate the chambers of the GKZ-fan by  $\sigma_1, ..., \sigma_r$.  Hence, for $1 \leq p \leq r$ we may choose a character $\chi_p$ in the interior of $\sigma_p$, and consider the semi-stable points with respect to that character.  This yields an open subset in $X$ which we denote by $U_p$ (it does not depend on the choice of character).  It also corresponds to a regular triangulation $\mathcal T_p$ of the collection of points $\{\nu_1(S), ..., \nu_{n+t}(S)\}$.

\begin{definition}
Let $\times : \gm^{n+t} \to \gm$ be the multiplication map.  We say that $S$ satisfies the \newterm{quasi-Calabi-Yau condition} if $\times|_{\widetilde{S}} = \bf{1}$, i.e., the multiplication map restricted to $\widetilde{S}$ is the trivial homomorphism.
\end{definition}

\begin{definition}
 Let $G$ be a group acting on a space $X$ and let $f$ be a global function on $X$.  We say that $f$ is \newterm{semi-invariant} with respect to a character $\chi$ if, for any $g \in G$,
 \[
 f(g \cdot x) = \chi(g)f(x).
 \]
 Equivalently, this means that $f$ is a section of the equivariant line bundle $\O(\chi)$ on the global quotient stack $[X/G]$.
\end{definition}

\begin{remark}
Each variable $x_i$ is semi-invariant with respect to a unique character of $S$ which we can denote by $\op{deg}(x_i)$.   The quasi-Calabi-Yau condition is equivalent to
\begin{equation}
\sum \op{deg}(x_i)+ \sum \op{deg}(u_j)
\end{equation}
being torsion.
\end{remark}

To apply Corollary \ref{cor: compact RG-flow}, we will add an auxiliary $\gm$-action and an $S$-invariant function which is $\gm$-semi-invariant.
This auxiliary $\gm$-action  acts with weight $0$ on the $x_i$ for all $i$ and with weight $1$ on the $u_j$  for  all $j$.  We refer to this auxiliary action as \newterm{$R$-charge}.

The action of $S$ on $\op{Spec} \kappa[u_j]$ gives a character $\gamma_j$ of $S$.
Let $f_1, ..., f_t$ be $S$-semi-invariant functions in the $x_i$ with respect to the character $\gamma_j^{-1}$. The functions $f_i$ determine a complete intersection in $\A^n$ as their common zero-set.   We can also use them to
define a function
\[
w := \sum_{j=1}^t u_j f_j.
\]
we call the \newterm{superpotential}.

The superpotential $w$ is $S$-invariant and $\chi$-semi-invariant for the projection character $\chi: S \times \gm \to \gm$.  This means that it is homogeneous of degree 0 for the $S$-action and homogeneous of degree $1$ with respect to the $R$-charge.  

Let $Z(w)$ denote the zero-locus of $w$ in $X$ and
\[
Y_p := Z(w) \cap U_p.
\]

\begin{theorem}[Herbst-Walcher]
If $S$ satisfies the quasi-Calabi-Yau condition, there is an equivalence of categories
\[
\dsing{[Y_p/ S \times \gm]} \cong \dsing{[Y_q/ S \times \gm]}
\]
for all $1 \leq p, q \leq r$.
\label{thm: HW}
\end{theorem}
\begin{proof}
This is essentially Theorem 3 of \cite{HW} stated in geometric as opposed to algebraic language.  For the geometric translation see Theorem 5.2.1 of \cite{BFK12} (version 2 on arXiv) or \cite{HL12} Corollary 4.8 and Proposition 5.5.
\end{proof}

We now refocus our attention to decribe explicitly the open sets $U_p\subseteq X$ corresponding to the semistable loci associated to the characters $\chi_p$ in $\op{Hom}(\widetilde{S}, \gm) {\otimes_{\Z}} \Q$. For $1 \leq p \leq r$, we can define the irrelevant ideal $\I_p$ that is associated to the character $\chi_p$ in the chamber $\sigma_p$ of the secondary fan as
\begin{equation*}
\I_p := \left\langle \prod_{i \notin I} x_i \prod_{j \notin J}u_j \middle| \ \bigcap_{i \in I} F_{i, \chi_p} \cap \bigcap_{j \in J} F_{n+j, \chi_p} \neq \emptyset \right\rangle
\end{equation*}
where $I \subseteq \{1, ..., n\}, J \subseteq \{1, ..., t\}$ and $F_{\chi_p}$ are the virtual facets of the polyhedron $P_{\chi_p}$ (see Sections 14.2 and 14.4 of \cite{CLS}). 

Alternatively, $\I_p$ can be defined by $\mathcal T_p$, the corresponding triangulation of $\nu(S)$.  Namely,
\begin{equation}
\I_p =  \left\langle \prod_{i \notin I} x_i \prod_{j \notin J} u_j  \middle| \ \bigcup_{i\in I} \nu_i(S) \cup \bigcup_{j\in J} \nu_{n+j}(S) \text{ is the set of vertices of a simplex in } \mathcal T_p \right\rangle.
\label{eq: I}
\end{equation}
The open set $U_p$ is complement of the zero set of the irrelevant ideal $\I_p$, i.e.,
\[
U_p = X \setminus Z(\I_p).
\]

We also consider a certain subideal of the irrelevant ideal given by taking all generators found by fixing $J = \{1, ..., t \}$:
\begin{align}
\J_p & := \left\langle \prod_{i \notin I} x_i  \middle| \ \bigcap_{i \in I} F_{i, \chi_p} \cap \bigcap_{j=1}^t F_{n+j, \chi_p} \neq \emptyset \right\rangle \notag \\
& =  \left\langle \prod_{i \notin I} x_i \middle| \ \bigcup_{i\in I} \nu_i(S) \cup \bigcup_{j=1}^t \nu_{n+j}(S) \text{ is the set of vertices of a simplex in }  \mathcal T_p \right\rangle.
\label{eq: J}
\end{align}

Note that the ideal $\J_p$ is non-zero if and only if there exists a simplex in the triangulation $\T_p$ whose set of vertices include $\nu_{n+j}(S)$ for all $j$. The complement of the zero-locus of $\J_p$ gives a new open set
\[
V_p := X \setminus Z(\J_p) \subseteq U_p.
\]

We may also view $\J_p$ as an ideal in $\kappa[x_1, ...,x_n]$ in which case we denote it by $\J_p^x$.   Now, restrict the action of $S$ to $\A^n = \op{Spec} \kappa[x_1, ..., x_n]$ (considered as a plane in $\A^{n+t}$).
This gives an open set of $\A^n$
\[
V_p^x := \A^n \setminus Z(\J_p^x)
\]
and a toric Deligne-Mumford stack
\[
X_p := [V_p^x / S].
\]

When $\J_p$ is a non-zero ideal, then $V_p^x$ and $X_p$ are nonempty and we can show that $[V_p/S]$ is a vector bundle over $X_p$. The inclusion of rings $\kappa[x_1, ..., x_n] \to \kappa[x_1, ..., x_n, u_1, ..., u_t]$ restricts to a $S$-equivariant morphism
\[
[V_p/S] \to [V_p^x / S] = X_p.
\]

\begin{proposition}
Suppose the ideal $\J_p$ is non-zero. The morphism
\[
[V_p/S] \to X_p.
\]
realizes $[V_p/S]$ as the total space of a vector bundle
\[
[V_p/S] \cong \op{tot} \bigoplus_{j=1}^t \O(\gamma_j).
\]
Furthermore, the $R$-charge action of $\gm$ is the dilation action along the fibers.
Finally, for each $j$, the function $f_j$ gives a section of $\O(\gamma_j^{-1})$ and the superpotential $w = \sum u_j f_j$ restricts to the pairing with the section $\bigoplus f_j$.
\label{prop: vec indentification}
\end{proposition}

\begin{proof}
Notice first that the open set $V_p$ decomposes as a product
\[
V_p = V_p^x \times \op{Spec} \kappa[u_1, ..., u_t].
\]
It is then a standard fact that the stack
\[
[V_p/S] = [V_p^x \times \op{Spec} \kappa[u_1, ..., u_t] / S]
\]
can be realized as the equivariant bundle on $[V_p^x / S]$ given by the representation of $S$ on $\op{Spec} \kappa[u_1, ..., u_t]$.

Now, the group $S$ acts on $\op{Spec} \kappa[u_1, ..., u_t]$ via the characters $\gamma_j$ and the representation is nothing more than the diagonal action of these characters.  Hence, we get precisely the statement
\begin{equation}
[V_p/S] \cong \op{tot} \bigoplus_{j=1}^t \O(\gamma_j).
\label{eq: total space realization}
\end{equation}

By definition, the R-charge action of $\gm$ acts with weight $0$ on $V_p^x$ and weight $1$ on $\op{Spec} \kappa[u_1, ..., u_t]$, i.e., by scaling on the second factor.    Under the isomorphism \eqref{eq: total space realization}, this $\gm$ just acts with weight $1$ along the fibers of the vector bundle, as desired.

Finally, by definition,
\[
\op{tot} \bigoplus_{j=1}^t \O(\gamma_j) = \underline{\op{Spec}}\left(\op{Sym}\left(\bigoplus_{j=1}^t \O(\gamma_j^{-1})\right)\right)
\]
with global functions identified as
\[
\op{H}^0\left(\op{Sym}\left(\bigoplus_{j=1}^t \O(\gamma_j^{-1})\right)\right) = \bigoplus_{j = 1, ..., t, r \in \Z} u_j^r\op{H}^0\left(\bigoplus_{j=1}^t \O(\gamma_j^{-r})\right)
\]
so that $w = \sum u_j f_j$ is identified with $\bigoplus f_j \in \op{H}^0(\bigoplus_{j=1}^t \O(\gamma_j^{-1})) \subseteq \op{H}^0(\op{Sym}(\bigoplus_{j=1}^t \O(\gamma_j^{-1})))$ as desired.
\end{proof}

From Proposition \ref{prop: vec indentification}, we see that for all $p$, the zero set of $\oplus f_j$ as a section of $V_p$ defines a complete intersection
\[
Z_p := Z(\oplus f_j) \subseteq X_p.
\]
We can also consider the zero locus of $w|_{U_p}$ which we denote by
\[
Y_p := Z(w) \cap U_p.
\]

Let $\partial w$ be the Jacobian ideal, i.e., the ideal generated by the partial derivatives of $w$ with respect to the $x_i$ and the $u_j$.
\begin{proposition}
Suppose the ideal $\J_p$ is non-zero. If $\I_p \subseteq \sqrt{\partial w, \J_p}$ then
\[
\dsing{[Y_p/ S \times \gm]} \cong \dbcoh{Z_p}.
\]
\label{prop: RG Flow}
\end{proposition}
\begin{proof}
Since, $\I_p \subseteq \sqrt{\partial w, \J_p}$ this implies that the singular locus of $w|_{U_p}$ is contained in $V_p$.  By Proposition~\ref{prop: vec indentification} we may apply Corollary~\ref{cor: compact RG-flow} with $X = Y_p$ and $U = Y_p \cap V_p$ to obtain the result.
\end{proof}

\begin{corollary}
Assume that $S$ satisfies the quasi-Calabi-Yau condition and that $\J_p$ and $\J_q$ are non-zero.  If $\I_p \subseteq \sqrt{\partial w, \J_p} \tand \I_q \subseteq \sqrt{\partial w, \J_q}$ for some $1 \leq p,q \leq r$,
then
\[
\dbcoh{Z_p} \cong \dbcoh{Z_q}.
\]
\label{cor: derivedequiv}
\end{corollary}
\begin{proof}
We have
\begin{align*}
 \dbcoh{Z_p} & \cong  \dsing{[Y_p/ S \times \gm]} \\
 & \cong  \dsing{[Y_q/ S \times \gm]} \\
 & \cong \dbcoh{Z_q},
   \end{align*}
   where the first line is Proposition~\ref{prop: RG Flow}, the second line is Theorem~\ref{thm: HW}, and the third line is Proposition~\ref{prop: RG Flow} again.
\end{proof}

\begin{remark}
For each $p$, the condition that $\I_p \subseteq \sqrt{\partial w, \J_p}$ is a locally closed condition on the set of $t$-tuples $f_j$ of $S$-invariant functions.  Hence, given two partial compactifications of vector bundles related by GIT, there is a locally-closed family of zero-sections of each bundle which are derived equivalent.
\end{remark}

\begin{remark}
For a single wall-crossing in the GKZ fan of a toric variety, one can look at the corresponding wall crossing in the GKZ fan of the total space of the canonical bundle.  The condition that $\I_p \subseteq \sqrt{\partial w, \J_p} \tand \I_q \subseteq \sqrt{\partial w, \J_q}$ is then equivalent to the hypersurface $w$ being nonsingular on the contracting loci.  These wall-crossings were first described independently by Dolgachev and Hu, and Thaddeus \cite{DH98, Tha96} and by Gel'fand, Kapranov, and Zelevinsky in the toric setting \cite{GKZ}.  For an explanation of terminology see \cite{BFK12}, especially Proposition 5.1.4 where the relevant contracting loci are described.
\label{rem: smooth version}
\end{remark}

\section{Derived Equivalence of Berglund-H\"ubsch-Krawitz Mirrors}
\label{sec: BHK}

Consider two quasihomogeneous polynomials $F_A$ and $F_{A'}$ with the same weights $q_i$.  Then,$J_{F_A} = J_{F_{A'}}$.  Fix a subgroup $J_{F_A} \subseteq SL(F_A) \cap SL(F_{A'})$ as in Equation~\eqref{groupGdef}.  Then, one can define the Calabi-Yau orbifolds $Z_{A, G}$ and $Z_{A', G}$ as well as the BHK mirrors $Z_{A^T, G^T_A}$ and $Z_{(A')^T, G^T_{A'}}$.

In this section, we prove the following theorem:

\begin{theorem}
Suppose we have two quasihomogeneous polynomials $F_A$ and $F_{A'}$ such that they have same weights $q_i$, i.e., $J_{F_A} = J_{F_{A'}}$. Choose a group $G$ so that $J_{F_A} \subset G \subset SL(F_A)$ and $J_{F_{A'}} \subset G\subset SL(F_{A'})$. Take $\overline{G} = G/J_{F_A}$. Then $Z_{A, G}$ and $Z_{A', G}$ are two hypersurfaces in $\P(q_0, \ldots, q_n) /\overline{G}$.  If the coarse moduli space of $\P(q_0, \ldots, q_n) / \overline{G}$ is Gorenstein Fano, then the BHK mirrors $Z_{A^T, G^T_A}$ and $Z_{(A')^T, G^T_{A'}}$ are derived equivalent.
\label{thm: BHK main result}
\end{theorem}

 \begin{proof}
This is the special case of Corollary~\ref{BHKDEQUIVGORFAM} where $b_i=1, c=0$.
 \end{proof}

The theorem above is a derived analogue to the birationality of Berglund-H\"ubsch-Krawitz mirrors (Theorem~\ref{thm: birationality BHK}).  We begin the proof by reducing to the case where $F_A - F_{A'}$ is a difference of two monomials.  For this reduction step, we introduce the notion of a Kreuzer-Skarke cleave.

\subsection{Kreuzer-Skarke Cleaves}
\label{sec: KS cleaves}

In this subsection, we show that if the coarse moduli space of $\P(q_0, \ldots, q_n) / \overline{G}$ is Gorenstein Fano, then for any two quasihomogeneous polynomials $F_A$ and $F_{A'}$ with the weights $q_i$, there is a sequence of quasihomogeneous polynomials, 
\[
F_A = F_{A_1}, ..., F_{A_t} = F_{A'}, 
\]
such that $F_{A_i} - F_{A_{i+1}}$ is a difference of two monomials.  This uses the classification of Kreuzer-Skarke polynomials i.e. quasihomogeneous, quasismooth potentials in $n+1$ variables with $n+1$ monomials terms:
 \begin{theorem} [Kreuzer-Skarke Classification \cite{KS92}]
Up to relabelling,  all Kreuzer-Skarke polynomials can be written as a sum of the following polynomials in separate variables:
\begin{enumerate}[i.]
\item Fermat: $W_{\text{fermat}} : = x^a$;
\item Loops of length $\ell >2$: $W_{\text{loop}} := x_1^{a_1} x_2 + x_2^{a_2}x_3 + \ldots + x_{\ell-1}^{a_{\ell-1}} x_\ell + x_\ell^{a_\ell}x_1$; and
\item Chains of length $\ell >2$: $W_{\text{chain}} : =  x_1^{a_1} x_2 + x_2^{a_2}x_3 + \ldots + x_{\ell-1}^{a_{\ell-1}} x_\ell + x_\ell^{a_\ell}$.
\end{enumerate}
\label{thm: KS repeat}
\end{theorem}
\noindent The polynomials in the list above are called \newterm{atomic types}.
 In the original Kreuzer-Skarke paper, the diagrams for such atomic types are the following
\begin{enumerate}
\item Fermat:
\begin{center}
\begin{tikzpicture}[
      scale=1,
      level/.style={->,>=stealth,thick}]
	\node (a) at (-1,0) {$\bullet_{a}$};
\end{tikzpicture}
\end{center}
\item Loop:
\begin{center}
\begin{tikzpicture}[
      scale=1,
      level/.style={->,>=stealth,thick}]
	\node (a) at (-5,0) {$\bullet_{a_1}$};
	\node (b) at (-3,0) {$\bullet_{a_2}$};
	\node (c) at (-1,0) {$\cdots$};
	\node (d) at (1.3,0) {$\bullet_{a_{\ell-1}}$};
	\node (e) at (3.3,0) {$\bullet_{a_\ell}$};
	\draw[level] (3,.2) .. controls (1,1) and (-3,1) .. (-5.1,.2) node at (-1,.95) {};
	\draw[->] (-5,.1) -- (-3.3,.1);
	\draw[->] (-3,.1) -- (-1.3,.1);
	\draw[->] (-.7,.1) -- (.8,.1);
	\draw[->] (1.3,.1) -- (3,.1);

\end{tikzpicture}
\end{center}

\item Chain:
\begin{center}
\begin{tikzpicture}[
      scale=1,
      level/.style={->,>=stealth,thick}]
	\node (a) at (-5,0) {$\bullet_{a_1}$};
	\node (b) at (-3,0) {$\bullet_{a_2}$};
	\node (c) at (-1,0) {$\cdots$};
	\node (d) at (1.3,0) {$\bullet_{a_{\ell-1}}$};
	\node (e) at (3.3,0) {$\bullet_{a_\ell}$};
	\draw[->] (-5,.1) -- (-3.3,.1);
	\draw[->] (-3,.1) -- (-1.3,.1);
	\draw[->] (-.7,.1) -- (.8,.1);
	\draw[->] (1.3,.1) -- (3,.1);

\end{tikzpicture}
\end{center}
\end{enumerate}

To each point in such a diagram, one can associate a monomial $x_i^{a_i}$ or $x_i^{a_i}x_j$ where $a_i$ is the weight at the vertex corresponding to $x_i$ and the factor $x_{j}$ depends on if there's an arrow pointing to the vertex corresponding to the variable $x_j$.  One obtains the three atomic types of polynomials by summing over vertices.  Hence, all Kreuzer-Skarke polynomials can be visualized as disjoint unions of the three types above.

\begin{remark}
If one takes the Kreuzer-Skarke diagram of a polynomial $F_A$, the Kreuzer-Skarke diagram of the transposed polynomial $F_{A^T}$ is the dual diagram resulting from reversing the direction of all the arrows.
\end{remark}

\begin{definition}
Consider two Kreuzer-Skarke polynomials $F_A$ and $F_{A'}$ defining hypersurfaces in $\P(q_0,\ldots, q_n) / \overline{G}$. Suppose that $F_A$ and $F_{A'}$ are related by deleting or adding a single arrow and changing the exponent $a_i$ at the source of the arrow.  In this case we say that the pair $(A, A')$ is a \newterm{Kreuzer-Skarke cleave}.
\end{definition}

\begin{definition}
Given an element $\textbf{b} \in \kappa^{n+1}$ and a diagram as above,
we define a \newterm{generalized Kreuzer-Skarke polynomial} as a polynomial of the form
\[
F_A^{\mathbf{b}} = \sum_{i=0}^n b_i p_i
\]
where $p_i = x_i^{a_i}x_j$ or $p_i = x_i^{a_i}$ according to the prescription above associated to the diagram.
\end{definition}

\begin{remark}
Given a Kreuzer-Skarke cleave $(A,A')$, we may also compare $F_A^{\mathbf{b}}, F_{A'}^{\mathbf{b}}$ for fixed $\textbf{b} \in \kappa^{n+1}$.
\end{remark}

We now digress into the toric interpretation of Kreuzer-Skarke cleaves. First let us recall some notation and review some standard facts. Let $M$ and $N$ be dual lattices, and let $\Sigma$ be the fan in $N_{\R}$ such that the associated toric stack $X_{\Sigma}$ is the fake weighted projective stack $\P(q_0, \ldots, q_n) / \overline{G}$.

 For any projective toric stack $X_\Sigma$ associated to a simplicial fan $\Sigma$, there is a polytope $\Delta \subset M_{\R}$ whose lattice points correspond to global sections of the anticanonical bundle on the stack. One can also construct a fan $\Sigma_K \subset (N\oplus \Z)_{\R}$ associated to the toric stack $X_{\Sigma_K} = \op{tot} \omega_{X_{\Sigma}}$ such that the support $|\Sigma_K|$ is a strictly convex cone. 
 
 To briefly outline its construction, for any ray $\rho \in \Sigma(1)$, we denote by $u_\rho \in N$ the corresponding primitive ray generator. Let $\rho_b$ be the ray in $(N\oplus\Z)_{\R}$ generated by the element $(0,1)$. The fan $\Sigma_K$ is the collection of cones
$$
\overline \sigma = \op{Cone}\left( (u_\rho, 1), (0,1) \ \middle| \ \rho \in \sigma(1)), (\sigma \in \Sigma\right)
$$
and their proper faces. By Lemma 5.17 of \cite{FK14}, we have that the dual cone to the support of the fan $\Sigma_K$ is $|\Sigma_K|^\vee = \cone (\Delta,1) \subset (M\oplus \Z)_{\R}$.

Denote by  $x_\rho$   the variable associated to the ray $\overline \rho := \op{Cone}((u_\rho, 1))$ in $\Sigma_K$ and $u$ the variable associated to the ray $\rho_b$. Recall from Equation~\eqref{latticepointsmonomial}  that the lattice point $\overline m : = (m,1) \in (\Delta,1)$ corresponds to the global function 
$$
x^{\overline m} := u^{\langle (m,1), (0,1)\rangle} \prod_\rho x_\rho^{\langle (m,1), (u_\rho,1)\rangle } = u \prod_\rho x_\rho^{\langle m, u_\rho\rangle + 1}
$$
on the toric variety $X_{\Sigma_K}$. By a minor abuse of notation, we also let $x_\rho$ correspond to the variable of $X_{\Sigma}$ associated to the ray $\rho \in \Sigma(1)$. Here the lattice point $m \in \Delta$ corresponds to the global section of the anticanonical bundle given by 
$$
x^m : = \prod_\rho x_\rho^{\langle m, u_\rho\rangle + 1}.
$$

A generalized Kreuzer-Skarke polynomial $F_A^{\mathbf{b}}$ can now be associated to a collection of $n+1$ lattice points \[
\Xi = \{\overline m_0,\ldots,\overline m_n\}  \subseteq (\Delta,1) \cap M \oplus \Z
\]
together with the $(n+1)$-tuple $\mathbf{b} \in \kappa^{n+1}$ so that the polynomial
$$
F_A^{\mathbf{b}} = \sum _{i=0}^n b_ix^{m_i}.
$$

\begin{lemma}
Let $q_0, ..., q_n$ be a weight sequence and $G$ be a finite group such that the coarse moduli space of $\P(q_0,\ldots, q_n) / \overline{G}$ is Gorenstein Fano.  Then, for each $i$, the monomial
$x_i^{\frac{d}{q_i}}$ is $G$-invariant.
\label{lem: GFermat}
\end{lemma}
\begin{proof}
Since the coarse moduli space of $\P(q_0,\ldots, q_n) / \overline{G}$ is Gorenstein Fano, the
anticanonical polytope $\Delta$ is reflexive, which means that the dual polytope $\Delta^\vee$ is also a lattice polytope. Consequently, the support of the fan for the canonical bundle on $\P(q_0,\ldots, q_n) / \overline{G}$ is the cone over $(\Delta^\vee,1)$. Moreover, the vertices of $(\Delta^\vee,1)$ are the minimal generators $(u_\rho,1)$ of the rays $\overline \rho \in \Sigma_K(1)$.  

By the dual cone correspondence, if $m$ is a vertex of the anticanonical polytope $\Delta$ then $\overline m = (m,1)$ pairs to 0 against a codimension 1 face of $\op{Cone}(\Delta^\vee,1)$.  Since $\Delta$ is a simplex, there exists exactly two rays in $\Sigma_K$ whose minimal generators pair positively with $\overline m$, one of which is $(0,1)$. This means that the global function associated with $ \overline m$ is
$$
x^{\overline m} = u x_\rho^{\langle u_\rho, m\rangle +1}
$$ 
for some $\rho \in \Sigma$, the vertex $m \in \Delta$ is associated to the Fermat monomial
$$
x^m= x_\rho^{\langle u_\rho, m\rangle +1}.
$$
Since the monomial $x_i^{a_i}$ is a section of the anticanonical, it is of degree $d$ hence $a_i = \frac{d}{q_i}$. 

\end{proof}

As a consequence of the lemma, given coefficients $\emph{\textbf{b}} \in \kappa^{n+1}$, we can define a $G$-invariant polynomial,
\[
\sum b_i x_i^{\frac{d}{q_i}},
\]
called the \newterm{generalized Fermat polynomial}.

\begin{proposition}\label{KSCleaveProp}
Fix $\emph{\textbf{b}} \in \kappa^{n+1}$. Take $d = \sum_i q_i$. Suppose the coarse moduli space of $\P(q_0,\ldots, q_n) / \overline{G}$ is Gorenstein Fano. Any $\overline{G}$-invariant generalized Kreuzer-Skarke polynomials of (weighted) degree $d$ with weights $q_0, ..., q_n$ is related to the generalized Fermat polynomial
\[
\sum b_i x_i^{\frac{d}{q_i}}
\]
by a sequence of Kreuzer-Skarke cleaves.
\end{proposition}
\begin{proof}

In any Kreuzer-Skarke diagram, a Fermat monomial term corresponds to a point having no outgoing arrow.  Start with any non-Fermat $G$-invariant monomial.  This corresponds to a point in the Kreuzer-Skarke diagram with an arrow coming out of it. A Kreuzer-Skarke cleave which deletes this arrow will give a $G$-invariant polynomial since, by Lemma~\ref{lem: GFermat}, the replaced monomial is $G$-invariant.  Delete all arrows in any order to get a sequence of Kreuzer-Skarke cleaves that relate $F_A$ with a Fermat polynomial.
\end{proof}

\subsection{Derived Equivalence of BHK Mirrors Related by a Kreuzer-Skarke Cleave}

We now will prove that if $F_A$ and $F_{A'}$ are related by a Kreuzer-Skarke cleave, then their BHK mirrors are derived equivalent.  The method is partially toric and will use results from Section 5 of \cite{FK14}.  We refer the reader there for a more thorough dictionary between the algebraic and toric language. A technical tool in the proof will be the use of two carefully chosen regular triangulations which correspond to different chambers for a given geometric invariant theory problem.

Recall that a Kreuzer-Skarke polynomial 
$$
F_A := \sum_{i=0}^n x^{m_i}
$$
corresponds to a collection of  $n+1$ lattice points
\[
\Xi = \{\overline m_0,\ldots,\overline m_n\}  \subseteq (\Delta,1) \cap M \oplus \Z.
\]
For notational clarity, we take $m_i$ to be the lattice point associated to the vertex $i$ in the Kreuzer-Skarke diagram, i.e., $x^{m_i}$ is of the form $x_i^{a_i}$ or $x_i^{a_i}x_j$ for some $j$. Now let $(A,A')$ be a Kruzer-Skarke cleave.  This gives us a new lattice element $\overline m_k' \in (\Delta,1) \cap (M\oplus \Z)$ defined by
$$
F_A - F_{A'} = x^{m_k'} - x^{m_k}.
$$

This gives two new point collections denoted by,
\[
\Xi' := (\Xi \backslash \{\overline m_k\}) \cup \{\overline m_k'\}.
\]
and
\[
\nu := \{(0,1), \overline m_0, ..., \overline m_n, \overline m_k' \} = \Xi \cup \Xi' \cup \{(0,1)\} \subseteq  (\Delta,1) \cap  M \oplus \Z.
\]

Recall the definition of a triangulationof $\nu$  (in $M \oplus \Z$):
\begin{definition}[\S 15.2 of \cite{CLS}]\label{defn:triangulation}
A \newterm{triangulation} of a point collection $\nu \in M \oplus \Z$ is a collection of simplices $\mathcal{T}$ so that: 
\begin{enumerate}
\item Each simplex in $\mathcal{T}$ is codimension one in $M_{\R} \oplus \R$ with vertices in $\nu$.
\item The intersection of any two simplices in $\mathcal{T}$ is a face of each.
\item The union of the simplices in $\mathcal{T}$ is $\op{Conv}(v \ | \ v \in \nu)$.
\end{enumerate}
\end{definition}

We can define two triangulations of $\nu$ (called $\mathcal T, \mathcal T'$) as follows. 
Let
$$
\mathcal{C} := \{ \op{Conv}(\{\overline m_i\}_{i\in I},(0,1)) \  |  \ I \subset \Xi, |I| = n\},
$$
i.e., $\mathcal C$ is the collection of simplices generated by any proper face of the convex hull of $n$ elements of the set $\Xi$ together with the element $(0,1)$.  Also, let
\[
 \mathcal S := \left\{ \operatorname{Conv}(\{\xi\}_{\xi \in I}, \overline m_k') | \  I\subseteq \Xi, |I| = n, \operatorname{Conv}(\{\xi\}_{\xi \in I}, \overline m_k') \cap \op{int}(\operatorname{Conv}(\Xi)) = \emptyset \right\}.
 \]
Note that elements of $\mathcal{S}$ need not be $n$-dimensional, hence we define the subset,
$$
\mathcal{S}_n := \{ T \in \mathcal{S} \ | \ \dim T = n\}.
$$
The triangulation $\mathcal T$ is built from $\mathcal C$ and $\mathcal S_n$,
\[
\T :=
\begin{cases}
\mathcal C & \tif \overline m_k' \in \operatorname{Conv}(\Xi) \\
\mathcal C \cup \mathcal S_n & \text{otherwise}. \\
\end{cases}
\]

We now define the triangulation $\mathcal T'$ analogously.  That is, we define
$$
\mathcal{C'} := \{ \op{Conv}(\{\overline m_i \}_{i\in I},(0,1))  \  |  \ I \subset \Xi', |I| = n\},
$$
$$
 \mathcal S' := \left\{ \operatorname{Conv}(\{\xi\}_{\xi \in I}, \overline m_k) \ \middle| \ I\subseteq \Xi',  |I| = n,  \operatorname{Conv}(\{\xi\}_{\xi \in I}, \overline m_k) \cap \op{int}(\operatorname{Conv}(\Xi')) = \emptyset \right.\},
 $$
 $$
\mathcal{S}_n' := \{ T \in \mathcal{S}' \ | \ \dim T = n\},
$$
\[
\T' :=
\begin{cases}
\mathcal C' & \tif \overline m_k \in \operatorname{Conv}(\Xi') \\
\mathcal C' \cup \mathcal S'_n & \text{otherwise}. \\
\end{cases}
\]

\begin{lemma}
Given a Kreuzer-Skarke cleave $(A,A')$ associated to anticanonical sections as above, the corresponding sets of simplices $\mathcal T, \mathcal T'$ are regular triangulations of $\nu$.
\label{lem: KS triangulations}
\end{lemma}
\begin{proof}
By Theorem 4 of \cite{Lee90}, all triangulations of point collections of cardinality at most $n+3$ whose corresponding convex hull is $n$-dimensional are regular.  Hence, it is enough to show that $\mathcal T, \mathcal T'$ are triangulations.  Since $\mathcal T, \mathcal T'$ are defined completely analogously, we only provide a proof for $\mathcal T$. 

To show that $\mathcal T$ is a triangulation, first, notice that $\op{Conv}(\overline m_i)$ is an $n$-dimensional simplex by using the fact that $A = (\langle \overline m_i, \overline u_{\rho_j}\rangle)_{ij}$ is an invertible matrix. Note that $\sum_i \overline m_i r_i = (0,d^T)$, using the fact that the transposed polynomial $F_{A^T}$ is quasihomogeneous of degree $d^T$  with positive weights $r_i$ so that $\sum_i r_i a_{ij}$ is equal to $d^T$ for all $j$. This implies that $(0,1)$ is in the relative interior of $\op{Conv}(\overline m_i)$. By definition, any $t \in \mathcal{C}$ is the convex hull of a facet of $\op{Conv}(\overline m_i)$ and $(0,1)$, hence is an $n$-dimensional simplex.  By definition, if $t \in \mathcal{S}_n$, then it is an $n$-dimensional simplex. All points in the set $\nu$ pair with $(0,1) \in (N\oplus \Z)_{\R}$ to one, so they are all in an affine hyperplane of $(M\oplus \Z)_{\R}$, making $\mathcal{T}$ satisfy the first condition in Definition~\ref{defn:triangulation}.

Second, it is easy to check that the intersection of any two simplices in $\mathcal T$ is given by the convex hull of the terms in $\nu$ they have in common hence a face of both simplices.  Checking this verifies the second condition in  Definition~\ref{defn:triangulation}.

Third, we need to check that
\[
\bigcup_{t \in \mathcal T} t = \op{Conv}(\nu).
\]
The containment $\bigcup_{t \in \mathcal T} t \subseteq \op{Conv}(\nu)$ is trivial.
The point $(0,1)$ is in the relative interior of the simplex $\op{Conv}(\Xi)$, hence
\[
\bigcup_{t \in \mathcal C} t = \op{Conv}(\Xi).
\]
If $\overline m_k' \in \op{Conv}(\Xi)$, then we are done. Otherwise, we note that $\op{Conv}(\nu)$ is the union of all lines between points in $\op{Conv}(\Xi)$ and $\overline m_k'$.  Take a generic point $p$ in $ \op{Conv}(\nu) \backslash \op{Conv}(\Xi)$. Take $q$ to be the point where the line from $\overline m_k'$ to $p$ first intersects the boundary of $\op{Conv}(\Xi)$.  The point $q$ lies on a facet $F$ of $\op{Conv}(\Xi)$ since $p$ is generic. Since $q$ lies on such a facet $F$ then $p$ is in $\op{Conv}(F, \overline m_k')$ which is in $\mathcal{S}_n$. We have now proven an open dense subset of $\op{Conv}(\nu)$ is contained in $\bigcup_{t \in \mathcal T} t$. Since all $t\in \mathcal{T}$ are closed, the union is closed, so we have the second containment $\bigcup_{t \in \mathcal T} t \supseteq \op{Conv}(\nu)$. This verifies the third condition of Definition~\ref{defn:triangulation}.
\end{proof}

Given a Kreuzer-Skarke polynomial $F_A$, a group $J_{F_A} \subseteq G \subseteq SL(F_A)$ and a vector $(c, \mathbf{b}) \in \kappa^{n+2}$ we can define a \newterm{generalized BHK pencil} to be the 1-parameter family of hypersurfaces
$$
Z_{A,G}^{c, \mathbf{b}} = \left[ \frac{\{F_A^{\mathbf{b}} + c \prod x_i = 0\} }{G \gm}\right] \subseteq \left[\frac{\A^{n+1} \setminus\{0\}}{G\gm}\right] = \frac{\P(q_0, \ldots, q_n) }{\overline{G}}.
$$

Any Kreuzer-Skarke cleave $(A,A')$, by definition, removes an arrow from the diagram for $F_A$ or $F_{A'}$.  The removal of an arrow always results in the formation of a new chain or Fermat diagram.  This chain or Fermat diagram has its tail at the head of the removed arrow.  Let $I$ be the indexing set which records the $a_i$ which this chain passes through.

 \begin{theorem}\label{BHKDEQUIV}
Suppose $(A,A')$ is a Kreuzer-Skarke cleave where $F_A, F_{A'}$ define anticanonical hypersurfaces in the fake weighted projective stack $\P(q_0, \ldots, q_n) / \overline{G}$.    If $b_i \neq 0$ for $i \in I$, then the generalized BHK mirror pencils $Z_{A^T, G^T_A}^{c, \mathbf{b}}$ and $Z_{(A')^T, G^T_{A'}}^{c, \mathbf{b}}$ are (memberwise) derived equivalent, i.e.,
 $$
 \dbcoh{Z_{A^T,G^T_A}^{c, \mathbf{b}}} \cong \dbcoh{Z_{(A')^T, G^T_{A'}}^{c,\mathbf{b}}}.
 $$
 \end{theorem}

 \begin{proof}
Recall we can write the two superpotentials $F_A$ and $F_{A'}$ as 
$$
F_A^{\mathbf{b}} := \sum_{i=1}^n b_i x^{m_i}
\tand
F_{A'}^{\mathbf{b}} := b_k x^{m_k'} + \sum_{i \neq k} b_i x^{m_i}.
$$
The CY orbifolds $Z_{A,G_A}^{c, \mathbf{b}}$ and $Z_{A', G_{A'}}^{c, \mathbf{b}}$ are hypersurfaces in the same toric variety $\P(q_0, \ldots, q_n)/\overline{G}$ and their BHK mirrors $Z_{A^T, G_A^T}^{c, \mathbf{b}}$ and $Z_{(A')^T, (G_{A'})^T}^{c, \mathbf{b}}$ are hypersurfaces in quotients of weighted projective stacks, say $\P(r_0, \ldots, r_n) / \overline{G_A^T}$ and $\P(r_0', \ldots, r_n') / \overline{(G_{A'})^T}$.  

\begin{caution}
As in Subsection~\ref{toricreinterpretation}, the CY-orbifolds $Z_{A^T, G^T_A}^{c, \mathbf{b}}$ and $Z_{(A')^T, G^T_{A'}}^{c, \mathbf{b}}$ are obtained by exchanging the roles of $M$ and $N$.  So now we will construct fans in $(M \oplus \Z)_{\R}$ whereas the standard in the toric literature is to construct fans in $N_{\R}$.
\end{caution}

Without loss of generality, the Kreuzer-Skarke cleave deletes an arrow, i.e., the monomial $x^{m_k'}$ is $x_k^{d/q_k}$ and the monomial $x^{m_k}$ is part of a loop or chain.
Recall the set of $n+3$ points 
\[
\nu := \{(0,1), \overline m_0, ..., \overline m_n, \overline m_k' \} = \Xi \cup \Xi' \cup \{(0,1)\} \subseteq
 (\Delta,1) \cap  M \oplus \Z
\]
and the two regular triangulations $\mathcal T, \mathcal T'$ of $\nu$ (see Lemma~\ref{lem: KS triangulations}).

Now, we have an exact sequence
\begin{align*}
0 \to N \oplus \Z & \to \Z^\nu \to \op{coker} \to 0 \\
\overline{n} & \mapsto \sum_{\rho \in \nu} \langle \overline{n}, \rho \rangle e_\rho.
\end{align*}
Let 
\[
S := \op{Hom}(\op{coker}, \gm)
\]
and define an action of $S$ on $\A^\nu$ given by the natural inclusion $S \hookrightarrow \gm^\nu$.
Notice that $S$ satisfies the quasi-Calabi-Yau condition by Lemma 5.12 of \cite{FK14} since $\nu$ lies in the affine plane $(M,1)$.

We set $y_i$ to be the coordinate of $\A^\nu$ associated to $\overline{m}_i$, $u$ to be the coordinate associated to $(0,1)$, and $y_k'$ to be the coordinate associated to $\overline{m}_k'$.  As both $\T$ and $\T'$ are regular triangulations of $\nu$, we get two irrelevant ideals $\I_p$ and $\I_q$ (as defined in Equation~\eqref{eq: I}) associated to the regular triangulations $\T$ and $\T'$ respectively.

Now, the generators of the subideals $\J_p \subseteq \I_p$ and $\J_q \subseteq \I_q$ as defined in Equation~\eqref{eq: J}  are in 1-1 correspondence to the simplices in $\mathcal{C}$ and $\mathcal{C}'$.  Notice that $\J_p$ and $\J_q$ are non-zero ideals as $\mathcal C, \mathcal C'$ are non-empty.  

Recall the open sets
  $$
  U_p = \op{Spec}(\kappa[y_0, \ldots, y_n, y_i', u]) \setminus Z(\I_p); \qquad U_q =   \op{Spec}(\kappa[y_0, \ldots, y_n, y_i', u]) \setminus Z(\I_q).
  $$
We also have the subsets
$$
V_p =  \op{Spec}(\kappa[y_0, \ldots, y_n, y_i', u]) \setminus Z(\J_p); \qquad V_q =   \op{Spec}(\kappa[y_0, \ldots, y_n, y_i', u]) \setminus Z(\J_q).
$$
The stacks $[V_p/S_\nu]$ and $[V_q/S_\nu]$ are the toric stacks that correspond to the fans $\Sigma_p$ and $\Sigma_q$ which are the collections of cones obtained by coning over the set of simplices in $\mathcal{C}$ and  $\mathcal{C}'$.  These correspond to the canonical bundles over the fake weighted projective stacks $\P(r_0, \ldots, r_n) / \overline{G_A^T}$ and $\P(r_0', \ldots, r_n') / \overline{(G_{A'})^T})$ respectively as described at the end of Subsection~\ref{toricreinterpretation}.

Now introduce the superpotential
$$
w := \sum_{i=0}^n b_i y^{\overline u_{\rho_i}} + c u  y^{(0,1)} = \sum_{i=0}^n b_i y^{\overline u_{\rho_i}} + c uy_k' y_0 \cdots y_n.
$$
and define
\[
 Z_p := Z(w) \subseteq X_p = \P(r_0, \ldots, r_n) / \overline{G_A^T}
 \]
  and
  \[
  Z_q := Z(w) \subseteq X_q = \P(r_0', \ldots, r_n') / \overline{G_{A'}^T}.
  \]
 When we take these zero loci, the polynomial $w$ specializes to only having the variables that correspond to the elements in $\Xi$ and $\Xi'$, respectively. By Equations~\eqref{mirmoninterp} and~\eqref{transposedpoltoric}, it follows that $w$ specializes to $F_{A^T}$ and $F_{(A')^T}$ respectively.  
 
 In summary, we have defined the two CY orbifolds
$$
Z_{A^T,G^T_A} = Z_p; \qquad Z_{(A')^T, G^T_{A'}}=Z_q.
$$
The derived equivalence desired now follows if we can use Corollary~\ref{cor: derivedequiv}.  In Lemma~\ref{IpInRadicalIdealLemma} below, we prove that the hypotheses of Corollary~\ref{cor: derivedequiv} hold, finishing the proof.
    \end{proof}

 \begin{lemma}\label{IpInRadicalIdealLemma}
Take the superpotential associated to the sum of the monomials corresponding to the lattice points $u_{\rho_i}$ that are the minimal generators of the rays in the fan $\Sigma$:
$$
w := \sum_{i=0}^n b_i y^{(u_{\rho_i}, 1)} + c u \prod y^{(0,1)}.
$$
If  $b_i \neq 0$ for all $i \in I$, then we have the following containment of ideals
 $$
 \I_p \subseteq \sqrt{\partial w, \J_p} \tand \I_q \subseteq \sqrt{\partial w, \J_q}.
 $$
 \end{lemma}

 \begin{proof}
 We use the notation in the previous proof.
 Take $F_A$ to be the sum of $\beta$ invertible polynomials of atomic types $F_{A_1}, \ldots, F_{A_\beta}$.  Without loss of generality, we say that $m_k$ is in $F_{A_1}$. Due to the assumption that $F_{A'}$ corresponds to having a Fermat term for the variable $x_k$, we know that $F_{A_1}$ must be either a chain or a loop.  We split our proof into these two cases as they give triangulations of a slightly different nature.

 {\bf Case 1: $F_{A_1}$ is a chain of length $\ell+1$.}

Since by assumption, $x_k^{a_{kk}}x_{k+1}$ is a summand of the atomic part $F_{A_1}$, we know that $k<\ell$.  We now look at the polytope $(\Delta, 1) \subseteq M_{\R} \times \R$. We have two triangulations $\T$ and $\T'$ as above. These triangulations correspond to irrelevant ideals $\I_p$ and $\I_q$ for some maximal chambers of the secondary fan corresponding to some characters $\chi_p$ and $\chi_q$.  The subideals of $\I_p$ and $\I_q$ generated by taking the monomials associated to the maximal simplices in the subcollections $\mathcal{C} \subseteq \T$ and $\mathcal{C'} \subseteq \T'$ yield the subideals $\J_p$ and $\J_q$ as in Equation~\eqref{eq: J}, namely,
 $$
 \J_p = \langle y_k' (y_0, \ldots, y_n)\rangle
$$
 and
 $$
 \J_q = \langle y_k(y_0, \ldots, y_{k-1}, y_k', y_{k+1}, \ldots, y_n)\rangle.
 $$
The quotients $\I_p/\J_p$ and $\I_q/\J_q$ are generated by the monomials associated to the simplices in the collections $\mathcal{S}$ and $\mathcal{S}'$ that are of maximal dimension.  While we need to prove that $\I_p \subseteq \sqrt{\partial w, \J_p}$ and $\I_q \subseteq \sqrt{\partial w, \J_q}$, we will instead prove something slightly stronger.  Namely
\begin{equation}
\I_p \subseteq  \left\langle y_k'(y_0, \ldots, y_n), u(y_{k+1}, \ldots, y_{\ell})\right\rangle  \subseteq \sqrt{\partial w, \J_p}
\label{eq: Ip}
\end{equation}
and
\begin{equation}
  \I_q \subseteq \left\langle y_k(y_0, \ldots, y_{k-1}, y_k', y_{k+1}, \ldots, y_n), u(y_{k+1}, \ldots, y_{\ell}) \right\rangle  \subseteq \sqrt{\partial w, \J_q}.
  \label{eq: Iq}
\end{equation}

We first establish the containments
\[
\I_p \subseteq  \left\langle y_k'(y_0, \ldots, y_n), u(y_{k+1}, \ldots, y_{\ell})\right\rangle\]
and
 \[
  \I_q \subseteq \left\langle y_k(y_0, \ldots, y_{k-1}, y_k', y_{k+1}, \ldots, y_n), u(y_{k+1}, \ldots, y_{\ell} )\right\rangle,
  \]
  from Equations~\eqref{eq: Ip} and~\eqref{eq: Iq}.
This is equivalent to showing that the simplices in $\mathcal S_n, \mathcal S'_n$ do not contain $(0,1)$ and some  $v \in \{ \overline m_{k+1}, ..., \overline m_{\ell} \}$.
The fact that each simplex in $\mathcal S_n, \mathcal S'_n$ does not contain $(0,1)$ and some other element $v \in \nu \setminus \{(0,1)\}$ is clear, so we now focus on proving that the $v$ dropped is in the set $\{\overline m_{k+1}, ..., \overline m_{\ell} \}$.

   The key observation is that the variables $\overline m_k', \overline m_k, \ldots, \overline m_{\ell}$ all live on the same $\ell - k - 1$ dimensional face of the polytope $(\Delta, 1)$. In particular, this is the face defined by taking the intersection of $(\Delta, 1)$ with the half spaces corresponding to the elements $(u_{\rho_i}, 1)$ for $k \leq i \leq \ell$, i.e.,
 $$
\overline m_k', \overline m_k, \ldots, \overline m_{\ell} \in (\Delta,1) \cap \bigcap_{ i \notin \{k , \ldots n\}} H_{(u_{\rho_i}, 1)}.
 $$
 This implies that one cannot obtain a simplex in $\mathcal S$ or $\mathcal S'$ whose set of vertices contains the set  $\{\overline m_k', \overline m_k, \ldots, \overline m_{\ell} \}$.  If $\overline m_k'$ is not a vertex, then we have $\op{Conv}(\Xi)$ and if $\overline m_k$ is not a vertex, then we have $\op{Conv}(\Xi')$.  Neither of these is in $\mathcal T, \mathcal T'$.
  This implies the desired containment.

We now establish the containments
\[
 \left\langle y_k'(y_0, \ldots, y_n), u(y_{k+1}, \ldots, y_{\ell})\right\rangle  \subseteq  \sqrt{ \partial w, \J_p}
 \]
  and
  \[
  \left\langle y_k(y_0, \ldots, y_{k-1}, y_k', y_{k+1}, \ldots, y_n), u(y_{k+1}, \ldots, y_{\ell}) \right\rangle \subseteq \sqrt{ \partial w, \J_q},
  \]
  from Equations~\eqref{eq: Ip} and~\eqref{eq: Iq}.

    It suffices to prove that the monomial $uy_j$ is in both ideals $ \sqrt{ \partial w, \J_q}$ and $ \sqrt{ \partial w, \J_p}$ for $k<j\leq \ell$.

First, one can describe all of the monomials of $w$ explicitly in terms of the matrix $A$
$$
y^{(u_{\rho_i}, 1)} = \begin{cases} y_0^{a_{00}}u & \text{ if $k\neq 0$ and $i=0$}. \\
						y_0^{a_{00}}(y'_0)^{b}u & \text{ if $k = i = 0$} \\
						y_i^{a_{ii}} y_{i-1} u & \text{if $0<i\leq \ell$, $i\neq k$} \\
						y_k^{a_kk}y_{k-1}(y_k')^{b}u  & \text{ if $0<k=i$} \\
						\prod_{j = \ell + 1}^n y_j^{a_{ji}}u & \text{ if $i >\ell$}.\end{cases}
$$
Note that $y_j$ does not divide the monomial $y^{(u_{\rho_i},1)}$ whenever $0\leq j \leq \ell$ and $i> \ell$.

We now take the partial derivative of $w$ with respect to the variable $y_k$ and consider,
$$
y_k \partial_k w = b_k a_{kk} y_k^{a_{kk}} y_{k-1} (y_k')^b u + b_{k+1} y_k y_{k+1}^{a_{(k+1)(k+1)}}u + c u y_k'\prod y_i.
$$
The first and third summand are in the ideals $\J_p, \J_q$.   Therefore $y_k y_{k+1} u$ is in the radical ideals $ \sqrt{ \partial w, \J_p}$ and $ \sqrt{ \partial w, \J_q}$ as $b_{k+1} \neq 0$ by assumption.

 Inductively, we now show that, provided that $y_{j-1}y_ju$ for $k<j<\ell$  is in $ \sqrt{ \partial w, \J_p}$ and $ \sqrt{ \partial w, \J_q}$, the monomial $y_{j+1}u$ is as well. We take the partial derivative with respect to $y_j$ of the superpotential $w$
$$
\partial_j w = b_j y_{j-1} y_j^{a_{jj}-1}u + b_{j+1} y_{j+1}^{a_{(j+1)(j+1)}}u + c u y_k' \prod_{i \neq j} y_i.
$$
The first and third summands are in  $ \sqrt{ \partial w, \J_p}$ and $ \sqrt{ \partial w, \J_q}$, consequently $y_{j+1}u$ is as well.

  Finally, return to the partial derivative and compute
\[
\partial_k w = b_k a_{kk} y_k^{a_{kk}-1} y_{k-1} (y_k')^b u + b_{k+1} y_{k+1}^{a_{(k+1)(k+1)}}u + c u y_k'\prod_{i \neq k} y_i.
\]
The first and third summands are in $\sqrt{ \partial w, \J_p}, \sqrt{ \partial w, \J_q}$ therefore $y_{k+1} u $ is as well.
This completes Case 1 as Equations~\eqref{eq: Ip} and~\eqref{eq: Iq} are satisfied.

 {\bf Case 2: $F_{A_1}$ is a loop of length $\ell+1$.}

 Similarly, we prove
\begin{equation}
\I_p \subseteq  \left\langle y_0'(y_0, \ldots, y_n), u(y_{1}, \ldots, y_{\ell})\right\rangle  \subseteq \sqrt{\partial w, \J_p}
\label{eq: Ip2}
\end{equation}
and
\begin{equation}
  \I_q \subseteq \left\langle y_0(y_0, \ldots, y_n), u(y_{1}, \ldots, y_{\ell}) \right\rangle  \subseteq \sqrt{\partial w, \J_q}.
\label{eq: Iq2}
\end{equation}

 As $F_{A_1}$ is a loop, without loss of generality we set $k = 0$. We apply a similar strategy to that of Case 1, but we have that $\overline m_0', \overline m_0, \ldots, \overline m_\ell$ all sit in the same face of $(\Delta,1)$, namely,
 $$
 (\Delta,1) \cap \bigcap_{j=\ell+1}^n H_{(u_{\rho_j},1)}.
 $$

The same argument gives the containments
\[
\I_p \subseteq  \left\langle y_0'(y_0, \ldots, y_n), u(y_{1}, \ldots, y_{\ell})\right\rangle
\]
and
\[
  \I_q \subseteq \left\langle y_0(y_0, \ldots, y_n), u(y_{1}, \ldots, y_{\ell}) \right\rangle
  \]
from   Equations~\eqref{eq: Ip2} and~\eqref{eq: Iq2}.

Again, one can explicitly describe the monomial terms of $w$ in terms of the matrix $A$,
$$
y^{(u_{\rho_i}, 1)} = \begin{cases} y_0^{a_{00}}(y_0')^{a_0'}y_{\ell}u & \text{if $i=0$}. \\
						y_{i-1}y_i^{a_{ii}}u & \text{ if $0< i \leq \ell$} \\
						\prod_{j = \ell + 1}^n y_j^{a_{ji}}u & \text{ if $j >\ell$}.\end{cases}
$$
We now prove the containments
\[
  \left\langle y_0'(y_0, \ldots, y_n), u(y_{1}, \ldots, y_{\ell})\right\rangle  \subseteq \sqrt{\partial w, \J_p}
\]
and
\[
\left\langle y_0(y_0, \ldots, y_n), u(y_{1}, \ldots, y_{\ell}) \right\rangle  \subseteq \sqrt{\partial w, \J_q},
\]
from   Equations~\eqref{eq: Ip2} and~\eqref{eq: Iq2}.

 First, take the partial derivative of $w$ with respect to $y_0$,
$$
y_0 \partial_0 w = b_0 a_{00} y_0^{a_{00}} (y_0')^b y_{\ell} u + b_1 y_0 y_1^{a_{11}}u + c u y_0'\prod y_i.
$$
As the first and third summands are in both $\J_p$ and $\J_q$, we know that $y_0y_1u$ is in both the radical ideals $\sqrt{ \partial w, \J_p}$ and $\sqrt{ \partial w, \J_q}$. We now can iterate the procedure.

 Given that the monomial $y_{j-1}y_j u$ is in both the ideals $\sqrt{ \partial w, \J_p}$ and $\sqrt{ \partial w, \J_q}$, we can prove that $y_{j+1} u$ is as well for $ 0< j < \ell$. Take the partial derivative with respect to $y_j$,
$$
\partial_j w = b_{j}a_{jj} y_{j-1}y_j^{a_{jj}} u + b_{j+1} a_{(j+1)(j+1)} y_{j+1}^{a_{(j+1)(j+1)}}u + c u y_0'\prod_{i \neq j} y_i
$$
as the first and third summands are in both ideals $\sqrt{ \partial w, \J_p}$ and $\sqrt{ \partial w, \J_q}$, we have that the second summand is as well, hence $y_{j+1}u$ is in both the radical ideals $\sqrt{ \partial w, \J_p}$ and $\sqrt{ \partial w, \J_q}$.

  Finally, return to the partial derivative at $y_0$,
$$
\partial_0 w = b_0 a_{00} y_0^{a_{00}-1} (y_0')^b y_{\ell} u + b_1 y_1^{a_{11}}u + c uy_0' \prod_{i \neq 0} y_i.
$$
The first and third summands are in $\sqrt{ \partial w, \J_p}, \sqrt{ \partial w, \J_q}$ therefore $y_1 u $ is as well.
This completes Case 2 as Equations~\eqref{eq: Ip2} and~\eqref{eq: Iq2} are satisfied.
\end{proof}

  \begin{corollary}\label{BHKDEQUIVCOR}
Suppose $(A,A')$ is a Kreuzer-Skarke cleave where $F_A, F_{A'}$ define hypersurfaces define anticanonical hypersurfaces in the fake weighted projective stack $\P(q_0, \ldots, q_n) / \overline{G}$. Then the BHK mirrors $Z_{A^T, G^T_A}$ and $Z_{(A')^T, G^T_{A'}}$ are derived equivalent, i.e.,
 $$
 \dbcoh{Z_{A^T,G^T_A}} \cong \dbcoh{Z_{(A')^T, G^T_{A'}}}.
 $$
 \end{corollary}
 \begin{proof}
 We set $c = 0$ and $b_i =1$ in Theorem~\ref{BHKDEQUIV}.
 \end{proof}

 \begin{corollary}
 \label{BHKDEQUIVGORFAM}
 Fix $\emph{\textbf{b}} \in (\kappa^*)^{n+1}, c \in \kappa$.  Take two polynomials $F_A$ and $F_{A'}$ which define hypersurfaces in a quotient of a Gorenstein Fano weighted projective stack $\P(q_0, \ldots, q_n) / \overline{G}$. Then the generalized BHK mirror pencils $Z_{A^T, G^T_A}^{c, \mathbf{b}}$ and $Z_{(A')^T, G^T_{A'}}^{c, \mathbf{b}}$ are derived equivalent.
 \end{corollary}
 \begin{proof}
Since we assume $b_i \neq 0$ for all $i$, this follows directly from iteratively using Theorem~\ref{BHKDEQUIV} to compare both $F_A$ and $F_{A'}$ through a sequence of Kreuzer-Skarke cleaves, which is guaranteed to exist by Proposition~\ref{KSCleaveProp}. \end{proof}

 \begin{remark}
Since $Z_{A^T,G^T_A}^{c, \mathbf{b}}, Z_{(A')^T,G^T_{A'}}^{c, \mathbf{b}}$ are open substacks of the irreducible component of the critical locus of $w$ lying on $Z(u)$, it follows that they are birational.  In the Gorenstein Fano case, this immediately recovers Theorem~\ref{thm: birationality BHK} in the case of families.
\label{rem: BHK birational}
\end{remark}

We can now rephrase Seidel and Sheridan's Homological Mirror Symmetry result for hypersurfaces in projective space  \cite{Sei03, Sheridan} in the language of Berglund-H\"ubsch-Krawitz mirror symmetry.
They define the \newterm{universal Novikov field} $\Lambda$, to be the field whose elements are formal sums
\[
\sum_{j=0}^{\infty} c_j r^{\lambda_j}
\]
where $c_j \in \C$, and $\lambda_j \in \R$ is an increasing sequence of real numbers such that
\[
\lim_{j \to \infty} \lambda_j = \infty.
\]
The universal Novikov field is algebraically closed of characteristic zero.

Over the universal Novikov field, we define a \newterm{Berglund-H\"ubsch-Krawitz pencil} as
$$
Z_{A, G}^{\op{pencil}} : = \left[ \frac{\{x_0 ... x_n + rF_{A} = 0\} }{G \gm}\right] \subseteq \left[\frac{\A^{n+1} \setminus\{0\}}{G \gm}\right] = \frac{\P(q_0, \ldots, q_n) }{\overline{G}}.
$$
Since Sheridan and Seidel have proven Homological Mirror Symmetry when $A^T$ is a Fermat polynomial, we obtain the following.
 \begin{theorem}
\label{HMS}
Homological Mirror Symmetry holds for Berglund-H\"ubsch-Krawitz mirror pencils in projective space over the universal Novikov field.

More precisely, if $F_A$ defines a smooth hypersurface in projective space $\P^n$ over the universal Novikov field (in particular $G=\Z_{n+1}$)  and $n \geq 3$, there is an equivalence of triangulated categories,
\[
\op{Fuk}{Z_{A, G}} \cong \dbcoh{Z_{A^T, G^T_A}^{\op{pencil}}}.
\]
\end{theorem}
\begin{proof}
Set $A' = (n+1)\op{Id}$, $G = J_{A'} = \Z_{n+1}$ and $q_0= ... = q_n = 1$.  We have
\begin{align*}
\op{Fuk}{Z_{A, G}} & = \op{Fuk}{Z_{A', G}} \\
& = \dbcoh{Z_{(A')^T, G^T_{A'}}^{\op{pencil}}}\\
& = \dbcoh{Z_{A^T, G^T_A}^{\op{pencil}}}.\\
\end{align*}
The first line follows from the fact that $Z_{A, G}$ is symplectomorphic to $Z_{A', G}$ by Moser's theorem.  The second line is Theorem 1.3 of \cite{Sei03} in the case $n=3$ and Theorem 1.2.7 of \cite{Sheridan} in the case $n \geq 4$.  The third line is Corollary~\ref{BHKDEQUIVGORFAM} in the special case $b_i =1$, $c = r$, and $\kappa = \Lambda$.
\end{proof}

\begin{remark}
In the case of elliptic curves ($n=2$), a variant of this theorem can be proven using work of Polishchuk and Zaslow \cite{PZ}.
\end{remark}

\begin{remark}
The category $\op{Fuk}{Z_{A, G}}$ is equipped with a $\Lambda$-linear structure and the equivalence is $\Lambda$-linear after changing the module structure of $\dbcoh{Z_{A^T, G^T_A}^{\op{pencil}}}$ by an automorphism of $\Lambda$.  See \cite{Sei03, Sheridan} for details.  It can then be extended to an equivalence of dg-categories using Theorem 9.8 of \cite{LO}.
\end{remark}

\subsection{An Example}

In the following example, we will see that our proof extends to families as well.

 \begin{example} Consider the polynomials $F_{A} = x_0^3 + x_1^2x_2 + x_2^3$ and $F_{A'} = x_0^3 + x_1^3 +x_2^3$. Both carve out cubic curves in $\P^2$. Let us take the fan of $\P^2$ which is the complete fan in $N_{\R} = (\Z)^2 \otimes \R$ generated by rays $(1,0)$, $(0,1)$ and $(-1,-1)$ and enumerate these rays as $x_{(1,0)} =: x_0$, $x_{(0,1)} =:x_1$ and $x_{(-1,-1)} =:x_2$ respectively. The canonical bundle of $\P^2$ is the toric variety associated to the fan $\Sigma_K$ which is defined to be the fan with rays generated by $u_{\rho_0} = (1,0,1)$, $u_{\rho_1} = (0,1,1)$ $u_{\rho_2} = (-1,-1,1)$ and $u_{\rho_3} = (0,0,1)$ and is the star subdivision along $\rho_3$ of the fan generated by $\rho_0, \rho_1$, and $\rho_2$.

 The dual cone to $|\Sigma_K|$ is generated by the elements $(2,-1,1)$, $(-1,2,1)$, and $(-1,-1,1)$. The polytope $\Delta$ that is associated to $\P^2$ is found by looking at the one slice $|\Sigma_K|_{(1)} = (\Delta, 1)$.  Note that since each lattice point corresponds to a monomial we can look at which lattice points correspond to monomials that are nonzero in $F_A$ and $F_{A'}$.

 \begin{figure}[h]
 \[
\begin{picture}(125,125)
\put(10,10){\circle*{5}}
\put(10,43){\circle{5}}
\put(10,76){\circle*{5}}
\put(10,109){\circle*{5}}
\put(43,10){\circle{5}}
\put(43,43){\circle*{5}}
\put(43,76){\circle{5}}
\put(76,10){\circle{5}}
\put(76,43){\circle{5}}
\put(109,10){\circle*{5}}
\put(10,10){\line(1,0){99}}
\put(10,10){\line(0,1){99}}
\put(10,109){\line(1,-1){99}}
\put(13,15){\small $x_2^3$}
\put(12,66){\small $x_1^2x_2$}
\put(20,104){\small $x_1^3$}
\put(99,20){\small $x_0^3$}
\end{picture}
\]
\caption{The polytope $\Delta$ with lattice points marked by sections of $\omega_{\P^2}$. }\label{fig: triangle}
\end{figure}
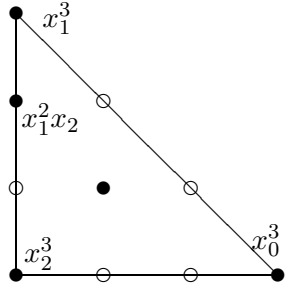

To consider the BHK mirrors, we set
\[
\nu := \{v_{\tau_0}, v_{\tau_1}, v_{\tau_1'}, v_{\tau_2}, v_{\tau_3}\}
\]
where $v_{\tau_0} = (2,-1,1)$, $v_{\tau_1} = (-1,2,1)$, $v_{\tau_1'}= (-1,1,1)$, $v_{\tau_2} = (-1,-1,1)$ and $v_{\tau_3} = (0,0,1)$.
We introduce variables for each ray: $y_i$ for $\tau_i$ where $i \in \{0,1,2\}$, $y_1'$ for $\tau_1'$ and $u$ for $\tau_3$. The triangulations $\T, \T'$ are pictured in Figure~\ref{fig: triangulations}.

\begin{figure}[h]
\[
\begin{picture}(375,125)
\put(10,10){\circle*{5}}
\put(10,43){\circle{5}}
\put(10,76){\circle*{5}}
\put(10,109){\circle*{5}}
\put(43,10){\circle{5}}
\put(43,43){\circle*{5}}
\put(43,76){\circle{5}}
\put(76,10){\circle{5}}
\put(76,43){\circle{5}}
\put(109,10){\circle*{5}}
\put(10,10){\line(1,0){99}}
\put(10,10){\line(0,1){99}}
\put(10,109){\line(1,-1){99}}
\put(10,76){\line(1,-1){33}}
\put(10,76){\line(3,-2){99}}
\put(43,43){\line(2,-1){66}}
\put(10,10){\line(1,1){33}}
\put(0,15){\small $y_2$}
\put(0,65){\small $y_1'$}
\put(0,100){\small $y_1$}
\put(109,20){\small $y_0$}
\put(43,33){\small $u$}
\put(62,100){$\T$}
\put(312,100){$\T'$}

\put(260,10){\circle*{5}}
\put(260,43){\circle{5}}
\put(260,76){\circle*{5}}
\put(260,109){\circle*{5}}
\put(293,10){\circle{5}}
\put(293,43){\circle*{5}}
\put(293,76){\circle{5}}
\put(326,10){\circle{5}}
\put(326,43){\circle{5}}
\put(359,10){\circle*{5}}
\put(260,10){\line(1,0){99}}
\put(260,10){\line(0,1){99}}
\put(260,109){\line(1,-1){99}}
\put(293,43){\line(2,-1){66}}
\put(260,109){\line(1,-2){33}}
\put(260,10){\line(1,1){33}}
\put(250,15){\small $y_2$}
\put(250,65){\small $y_1'$}
\put(250,100){\small $y_1$}
\put(359,20){\small $y_0$}
\put(293,33){\small $u$}
\end{picture}
\]
\caption{The triangulations $\mathcal T, \mathcal T'$ of $\nu$. }\label{fig: triangulations}
\end{figure}
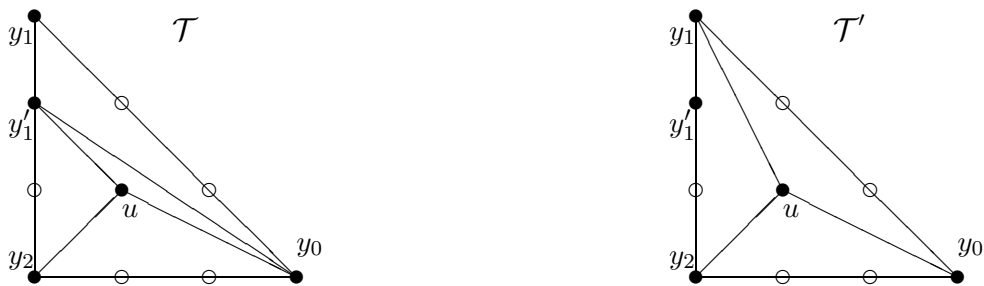
The corresponding irrelevant ideals are
 $$
 \I_p = \langle y_1(y_0,y_1',y_2), uy_2\rangle = \langle y_1y_0, y_1y_1', y_1y_2, uy_2\rangle \text{ and } \I_q = \langle y_1'(y_0,y_1,y_2)\rangle=\langle y_1'y_0, y_1'y_1, y_1'y_2\rangle,
 $$
 respectively.

 There exists subideals $\J_p =  \langle y_1(y_0,y_1',y_2)\rangle = \langle y_1y_0, y_1y_1', y_1y_2 \rangle$ and $\J_q = \I_q$ which correspond to the fans over the triangulations in Figure~\ref{fig: subtriangulations}.
 The toric varieties associated to $\Xi$ and $\Xi' $ are  $\op{tot}(\omega_{\P(2,3,1)})$ and $\op{tot}(\omega_{\P^2 / \Z_3})$ respectively.

 \begin{figure}[h]
 \[
\begin{picture}(375,125)
\put(10,10){\circle*{5}}
\put(10,43){\circle{5}}
\put(10,76){\circle*{5}}
\put(43,10){\circle{5}}
\put(43,43){\circle*{5}}
\put(76,10){\circle{5}}
\put(109,10){\circle*{5}}
\put(10,10){\line(1,0){99}}
\put(10,10){\line(0,1){66}}
\put(10,76){\line(1,-1){33}}
\put(10,76){\line(3,-2){99}}
\put(43,43){\line(2,-1){66}}
\put(10,10){\line(1,1){33}}
\put(0,15){\small $y_2$}
\put(0,65){\small $y_1'$}
\put(109,20){\small $y_0$}
\put(43,33){\small $u$}
\put(62,100){$\Xi$}
\put(312,100){$\Xi'$}
\put(260,10){\circle*{5}}
\put(260,43){\circle{5}}
\put(260,76){\circle{5}}
\put(260,109){\circle*{5}}
\put(293,10){\circle{5}}
\put(293,43){\circle*{5}}
\put(293,76){\circle{5}}
\put(326,10){\circle{5}}
\put(326,43){\circle{5}}
\put(359,10){\circle*{5}}
\put(260,10){\line(1,0){99}}
\put(260,10){\line(0,1){99}}
\put(260,109){\line(1,-1){99}}
\put(293,43){\line(2,-1){66}}
\put(260,109){\line(1,-2){33}}
\put(260,10){\line(1,1){33}}
\put(250,15){\small $y_2$}
\put(250,100){\small $y_1$}
\put(359,20){\small $y_0$}
\put(293,33){\small $u$}
\end{picture}
\]
\caption{The triangulations $\Xi, \Xi'$. }\label{fig: subtriangulations}
\end{figure}
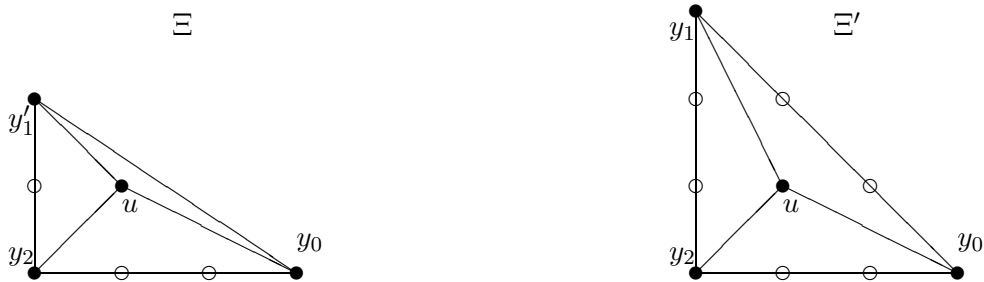

 We now need to discuss the superpotential $w$ that is a function on the partial compactifications of these bundles. To do this, we must turn back to the dual cone to $\text{Cone}(\nu)$. In this case, the dual cone is just $|\Sigma_K|$ (on a general Gorenstein Fano quotient of weighted projective space, the dual cone contains $|\Sigma_K|$ with equality if and only if $F_A$ or  $F_{A'}$ is a Fermat polynomial).  We draw the support of the dual cone $|\Sigma_K|_{(1)}$ below along with the functions corresponding to the lattice points in Figure~\ref{fig: labelled monomials}.
\begin{figure}[h]
 \[
 \begin{picture}(125,125)
 \put(10,10){\circle*{5}}
 \put(60,60){\circle*{5}}
 \put(60,110){\circle*{5}}
 \put(110,60){\circle*{5}}
 \put(10,10){\line(2,1){100}}
 \put(10,10){\line(1,2){50}}
 \put(60,110){\line(1,-1){50}}
 \put(10,0){\small $y_1'y_2^3 u$}
 \put(45,65){\small $y_0y_1y_1'y_2u$}
 \put(66,110){\small $y_1^3(y_1')^2u$}
 \put(110,65){\small $y_0^3 u$}
 \end{picture}
 \]
 \caption{Functions corresponding to lattice points in $\Delta^\vee$. }\label{fig: labelled monomials}
\end{figure}
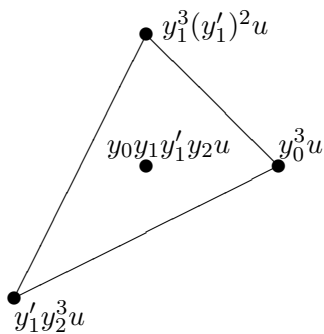

Now, let
$$
w := c_0y_0^3 u + c_1y_1^3(y_1')^2u + c_2y_1'y_2^3 u + c_3y_0y_1y_1'y_2u
$$
for some constants $c_i \in \kappa$. We need to check that we have that $\I_p \subseteq \sqrt{\partial w, \J_p}$ in order to be able to use Corollary \ref{cor: derivedequiv} (as $\I_q =\J_q$ this is automatic for the other triangulation).  Here, we compute the partial derivative of $w$ with respect to $y_1'$:
$$
\partial_{y_1'}w = 2 c_1y_1^3 (y_1')u + c_2 y_2^3 u + c_3 y_0y_1y_2u.
$$
Here we can see that the first and third summands are both in $\J_p$, hence $y_2u$ is in $\sqrt{\partial w, \J_p}$ as long as the constant $c_2$ is nonzero. In other words, one can apply Corollary \ref{cor: derivedequiv} as long as $c_2$ is nonzero.  Applying the framework outlined in Section~\ref{sec: algebraic}, we get:

\begin{equation}\begin{aligned}
U_p &:= \A^5\setminus Z(\I_p);  &  U_q &:= \A^5 \setminus Z(\I_q); \\
V_p &:= \A^5 \setminus Z(\J_p);  & V_q &:= \A^5 \setminus Z(\J_q); \\
V_p^x &:= \A^4 \setminus Z(\J_p^x);  & V_q^x &:= \A^4 \setminus Z(\J_q^x); \\
[V_p/S] &:= \op{tot}(\omega_{\P(2,3,1)}); & [V_q/S] &:= \op{tot}(\omega_{\P^2/\Z_3}); \\
X_p := [V_p^x/S] &= \P(2,3,1); \qquad& X_q := [V_q^x/S] &= \P^2/\Z_3; \\
Z_p &:= Z(w_p); & Z_q &:= Z(w_q);
\end{aligned}\end{equation}
where
\begin{equation}\begin{aligned}
w_p &:= c_0y_0^3 + c_1(y_1')^2 + c_2y_1'y_2^3 + c_3y_0y_1'y_2;\\
w_q &:= c_0y_0^3 + c_1y_1^3 + c_2y_2^3  + c_3y_0y_1y_2.
\end{aligned}\end{equation}
Then we have the equivalence of categories $\dbcoh{Z_p} \cong \dbcoh{Z_q}$.

The special case $c_0=c_1=c_2 = 1$ and $c_3 = 0$ is $w_p = F_{A^T}$ and $w_q = F_{(A')^T}$, which gives us the BHK mirrors to $Z_A$ and $Z_{A'}$. If we take $c_0=c_1=c_2 = 1$ and $c_3 = \lambda$, we have pencils. Also, we can take degenerate loci, for example, $c_2=1$ and $c_0=c_1=c_3=0$ so that $w_p= y_1'y_2^3 $ and $w_q = y_2^3$.  In general, we have locally-closed BHK mirror families that are pointwise derived equivalent to one another.

 \end{example}

\end{document}